\documentclass[a4paper,11pt,reqno]{amsart}
\usepackage{amsmath,amsthm,amssymb}
\usepackage[nospace,noadjust]{cite}
\usepackage[dvipsnames]{xcolor}
\usepackage{enumitem}
\usepackage{booktabs}
\usepackage{aliascnt}
\usepackage{hyperref}
\hypersetup{
  colorlinks   = true,
  urlcolor     = blue,
  linkcolor    = Purple,
  citecolor    = red
}
\usepackage[capitalise]{cleveref}
\usepackage[margin=2.5cm]{geometry}
\usepackage[foot]{amsaddr}
\allowdisplaybreaks
\theoremstyle{definition}
\newtheorem{theorem}{Theorem}[section]
\crefname{theorem}{Theorem}{Theorems}
\newaliascnt{definition}{theorem}
\newtheorem{definition}[definition]{Definition}
\aliascntresetthe{definition}
\crefname{definition}{Definition}{Definitions}
\newaliascnt{example}{theorem}
\newtheorem{example}[example]{Example}
\aliascntresetthe{example}
\crefname{example}{Example}{Examples}
\newaliascnt{remark}{theorem}
\newtheorem{remark}[remark]{Remark}
\aliascntresetthe{remark}
\crefname{remark}{Remark}{Remarks}
\newaliascnt{corollary}{theorem}
\newtheorem{corollary}[corollary]{Corollary}
\aliascntresetthe{corollary}
\crefname{corollary}{Corollary}{Corollaries}
\newaliascnt{proposition}{theorem}
\newtheorem{proposition}[proposition]{Proposition}
\aliascntresetthe{proposition}
\crefname{proposition}{Proposition}{Propositions}
\newaliascnt{notation}{theorem}
\newtheorem{notation}[notation]{Notation}
\aliascntresetthe{notation}
\crefname{notation}{Notation}{Notations}
\newaliascnt{lemma}{theorem}
\newtheorem{lemma}[lemma]{Lemma}
\aliascntresetthe{lemma}
\crefname{lemma}{Lemma}{Lemmas}
\newaliascnt{problem}{theorem}

\aliascntresetthe{problem}
\crefname{problem}{Problem}{Problems}

\newcommand{\F}{\mathbb F}
\newcommand{\K}{\mathbb K}
\newcommand{\Fq}{\F_q}
\newcommand{\mC}{\mathcal C}
\newcommand{\mD}{\mathcal D}
\newcommand{\mL}{\mathcal L}
\newcommand{\mX}{\mathcal X}
\newcommand{\mE}{\mathcal E}
\newcommand{\mG}{\mathcal G}
\newcommand{\crit}{\mathrm{crit}}

\DeclareMathOperator{\supp}{supp}
\DeclareMathOperator{\Mat}{Mat}
\DeclareMathOperator{\PG}{PG}
\DeclareMathOperator{\codim}{codim}

\DeclareMathOperator{\rowsp}{rowsp}
\DeclareMathOperator{\colsp}{colsp}
\DeclareMathOperator{\wt}{wt}
\DeclareMathOperator{\rank}{rank}
\DeclareMathOperator{\im}{Im}
\DeclareMathOperator{\dd}{d}

\title{Additive codes arising from hypergraphs}
\author{Gianira N. Alfarano}
\address{\textnormal{Universit\'e de Rennes, IRMAR, Rennes, France.}}
\email{gianira-nicoletta.alfarano@univ-rennes.fr}

\begin{document}
\begin{abstract}
We study the critical exponent of additive codes through an integer polymatroid associated with the code. We give a coding-theoretic proof of Whittle's Critical Theorem in this setting, a geometric description of the critical exponent in terms of $h$-projective systems, and general bounds, including an analogue of Kung's girth bound. We then study additive codes whose polymatroid is the hypergraphic polymatroid of a hypergraph $H$. For these codes the critical exponent turns to be determined by the weak chromatic number of $H$. If the code is faithful, then the minimum folded Hamming weight of the dual code is equal to  the Berge girth of $H$. If $H$ is connected, the minimum distance is equal to  the edge-connectivity of $H$. As a consequence, for $h\geq2$ we determine all such codes with connected $H$ that attain the Singleton bound, that is, all faithful hypergraphic additive quasi-MDS codes. We specialise the Griesmer and linear programming bounds to hypergraphic codes, we derive a lower bound on the minimum distance from the Laplacian eigenvalues of the weighted $2$-section of $H$, and we compare all these bounds computationally.
\end{abstract}

\maketitle

\section{Introduction}

The characteristic polynomial is a well-studied invariant in matroid theory and its generalisations. Its relation with critical problems was developed by Crapo and Rota for projective geometries and by Whittle for representable polymatroids; see \cite{crapo1970foundations,whittle1994critical,kung1996critical}. This problem provides a more general setting for many results in extremal combinatorics
 and has been the subject of numerous applications and generalisations; see \cite{alon,britz2005extensions,bryl_oxley,zaslavsky1987mobius,koga+,gruica2022rank, imamura2023, alfarano2026recursive}. For a linear code, the critical exponent is the minimum dimension of a subcode whose support is the whole set of coordinates \cite{dowling1971codes,britz2016covering}. In this paper we consider the same problem for additive codes, that is, for $\F_q$-linear subspaces of $\F_{q^h}^n$, which recently gained a lot of attention due to their application in quantum error-correction. In particular, additive codes geometrically correspond to $h$-projective systems, that is, to collections of projective subspaces of $\PG(k-1,q)$ of projective dimension at most $h-1$, and their minimum distance is determined by the maximum number of these subspaces contained in a hyperplane. This
geometric approach has recently led to several results on the parameters of additive codes and to the discovery of new additive codes whose parameters outperform those of linear codes; see e.g.
\cite{kurz2024additive,d2026generalized,adriaensen2023additive,bartoli2025long}.

To an additive code $\mC\subseteq\F_{q^h}^n$ of $\F_q$-dimension $k$ we associate the integer polymatroid $P_{\mC}$ with ground set $[n]$ and rank function $\rho_{\mC}(A)=k-\dim_{\F_q}\mC(A)$, where $\mC(A)$ is the subcode of codewords vanishing on $A$. It is represented by the column spaces of the blocks of an expanded generator matrix of $\mC$. We give a direct proof of Whittle's Critical Theorem for polymatroids in \cref{thm:critical-additive}. This gives a geometric description of the critical exponent in terms of the~$h$-projective system of $\mC$ and of blocking sets in the sense of \cite{d2026generalized}, and leads to general bounds on the critical exponent, including a generalisation of Kung's girth bound \cite{kung1996critical}. In this bound the girth is replaced by the dual distance of the code; for $h=1$ it is Kung's bound, and for every $h$ it is attained by the codes from spreads; see \cref{thm:kung-block,ex:spread}.

We then consider hypergraphs. For a hypergraph $H$ with vertex set $V$, Whittle defined in \cite{whittle1992hypergraph} a hypergraphic polymatroid as $f_H(A)=|V|-\kappa(A)$, where $\kappa(A)$ is the number of connected components of the hypergraph on $V$ with edge set $A$; see also \cite{helgason1974,lovasz1977flats}. We give an explicit representation of $f_H$ over every field in \cref{thm:canonical-hypergraphic-representation}. We call an additive code \emph{hypergraphic} if $\rho_{\mC}=f_H$ for some hypergraph $H$, and we call such an $H$ an \emph{underlying hypergraph} of $\mC$. For an $r$-uniform hypergraph we construct a code $\mC_H\subseteq\F_{q^{r-1}}^{|\mE|}$ with underlying hypergraph $H$; see \cref{def:hypergraphic-code}. A hypergraphic code need not be of this form, and an underlying hypergraph need not be unique, so we state our results for arbitrary hypergraphic codes.

For a hypergraphic code with underlying hypergraph $H$, the Helgason-Whittle identity $\chi_H^{\mathrm w}(z)=z^{\kappa(H)}p_{\mC}(z)$ relates the weak chromatic polynomial of $H$ to the characteristic polynomial of $\mC$, and gives $\crit(\mC)=\lceil\log_q\chi_{\mathrm w}(H)\rceil$; see \cref{thm:chromatic-critical}. For $\mC_H$ this correspondence gives more. Indeed, a weak colouring of $H$ with $q^t$ colours corresponds to an ordered $t$-tuple of codewords with full joint support, and each such $t$-tuple arises from exactly $q^{t\kappa(H)}$ colourings; see \cref{prop:colorings-full-support}.

We finally consider the parameters of hypergraphic additive codes. We focus on \emph{faithful} codes, where for every coordinate $i$ the set $\{c_i:c\in\mC\}$ of $i$-th entries of codewords is the whole alphabet $\F_{q^h}$, following the definition from \cite{kurz2024additive}. Writing the codewords over $\F_q$ as vectors in $(\F_q^h)^n$, the Hamming weight becomes the folded Hamming weight introduced in \cite{martinez2025folded}, and the code has a dual in $(\F_q^h)^n$. For a faithful hypergraphic code, the minimum folded Hamming weight of the dual code is equal to  the Berge girth of every underlying hypergraph (\cref{thm:dual-berge-girth}), and for a connected underlying hypergraph $H$ the minimum distance is equal to  the edge-connectivity $\lambda(H)$ (\cref{thm:distance-connectivity}). For graphs these are well-known facts on cut and cycle codes \cite{hakimi1968graph}.

Together with the Moore bound for hypergraphs, these identities show that, for a given minimum distance, a large dual distance forces a large dimension; see \cref{prop:moore}. This girth-distance bound is attained by several codes; see \cref{cor:sts,tab:distance-hypergraphic}. Using these identities we determine all faithful hypergraphic additive codes with a connected underlying hypergraph that attain the Singleton bound, that is, all such quasi-MDS codes in the sense of \cite{martinez2025folded}; see \cref{thm:hypergraphic-singleton-classification}. For $h\geq2$ and normalised dimension $k/h>1$, a quasi-MDS code in this class with minimum distance at least $2$ either has minimum distance $n-1$ and length at most $h+2$, or has length $4$. Both cases occur over every finite field, and the codes of length $4$ correspond to regular multigraphs on $4$ vertices; see \cref{prop:four-edge-family}. We also combine the distance identity with the Griesmer bound for additive codes (\cite{ball2025griesmer,kurz2024additive}) and with Delsarte's linear programming bound (\cite{delsarte1973}), and we obtain a lower bound on $d(\mC)$ from the second smallest Laplacian eigenvalue of the weighted $2$-section of $H$; see \cref{thm:spectral-lower}.

Finally, in \cref{subsec:numerics} we compare these bounds computationally, on some hypergraphic codes and on known optimal additive codes \cite{kurz2024additive}. For hypergraphic codes, the bounds coming from the hypergraph are much stronger than the general coding bounds. For general additive codes, the linear programming bound improves on the Griesmer bound in a few cases and attains the known optimal length $n_2(8,2;3)=5$.

The paper is organised as follows. \Cref{sec:prelim} recalls lattices, polymatroids, the critical problem, additive codes, projective systems and hypergraphs. \Cref{sec:polymatroid} introduces the polymatroid of an additive code, expresses the minimum distance and the dual distance through it, and contains the Critical Theorem and bounds on the critical exponent. \Cref{sec:hypergraphic} studies hypergraphic codes and weak colourings. \Cref{sec:distance} relates the minimum distance to edge-connectivity, proves the girth-distance bound and classifies hypergraphic QMDS codes. \Cref{sec:lp} specialises the Griesmer and linear programming bounds to hypergraphic codes, contains a spectral lower bound on the distance, and compares all bounds computationally.

\section{Preliminaries}\label{sec:prelim}
In this section we recall the background material on lattices, polymatroids, the critical problem, additive codes and hypergraphs.

\begin{notation}
Throughout the paper, $q$ is a prime power and $h,n,k$ are positive integers. For a positive integer $s$ we write $[s]:=\{1,\ldots,s\}$, and for $a\geq0$ we set $[a]_q:=(q^a-1)/(q-1)$, so that $[0]_q=0$. For a field $\mathbb K$ and positive integers $a,b$, we denote by $\Mat_{a,b}(\mathbb K)$ the set of $a\times b$ matrices with entries in $\mathbb K$. For a matrix $M\in\Mat_{a,b}(\Fq)$, $\colsp_{\F_q}(M)$ denotes its column space over $\F_q$. For a vector subspace $U\subseteq\F_q^k$, we denote by $\PG(U)$ the projective space whose points are the $1$-dimensional subspaces of $U$, which has projective dimension $\dim U-1$. In particular, $\PG(k-1,q):=\PG(\F_q^k)$. Dimensions and codimensions are always those of vector spaces unless we say \emph{projective} dimension. A \textbf{hyperplane} of $\PG(k-1,q)$ is $\PG(\Lambda)$ for a vector hyperplane $\Lambda\subseteq\F_q^k$, and a point $\langle v\rangle$ lies in $\PG(\Lambda)$ if and only if $v\in\Lambda$. For $k=1$ the only vector hyperplane of $\F_q^1$ is $\{0\}$, and we regard $\PG(\{0\})=\emptyset$ as the unique projective hyperplane for $k=1$.
\end{notation}

\subsection{Lattices and the M\"obius function}
We recall some basic facts on finite posets and lattices; see \cite{Birkhoff,romanlattice,stanley2011enumerative}. A \textbf{lattice} is a partially ordered set in which any two elements $a,b$ have a least upper bound $a\vee b$ (their \textbf{join}) and a greatest lower bound $a\wedge b$ (their \textbf{meet}). For $a\leq b$ in a poset, the \textbf{interval} $[a,b]$ is the set of all $x$ with $a\leq x\leq b$.

\begin{definition}
Let $(\mathcal P,\leq)$ be a finite poset. The \textbf{M\"obius function} $\mu:\mathcal P\times\mathcal P\to\mathbb Z$ is defined recursively by
\begin{align*}
\mu(X,Y) := \left\{
        \begin{array}{cl}
        1 & \text{ if } X=Y,\\
        -\sum\limits_{{X\leq Z<Y}}\mu(X,Z) & \text{ if } X<Y, \\
        0 & \text{otherwise}.
        \end{array}
          \right.
\end{align*}
\end{definition}

If $E$ is a finite set, the power set $2^E$ ordered by inclusion is a lattice, the \textbf{Boolean lattice}, with $A\vee B=A\cup B$ and
$A\wedge B=A\cap B$. The subspaces of a finite-dimensional vector space $W$, ordered by inclusion, form the \textbf{subspace lattice}, with $U\vee U'=U+U'$ and $U\wedge U'=U\cap U'$.
In the Boolean lattice, $\mu(A,B)=(-1)^{|B|-|A|}$ if $A\subseteq B$ and $\mu(A,B)=0$ otherwise.
We have the following well-known M\"obius inversion formula.
\begin{lemma}
Let $f,g:\mathcal{P}\to\mathbb{Z}$ be functions on a finite poset $\mathcal{P}$. Then
\begin{enumerate}
		\item$\displaystyle f(X)=\sum\limits_{X\leq Y}g(Y)\:\forall\:X \in \mathcal P \mbox{ if and only if }g(X)=\sum\limits_{X\leq Y}\mu(X,Y)f(Y) \:\forall\:X \in \mathcal P$.
		\item$\displaystyle f(X)=\sum\limits_{X\geq Y}g(Y)\:\forall\:X \in \mathcal P\mbox{ if and only if }g(X)=\sum\limits_{X\geq Y}\mu(Y,X)f(Y)\:\forall\:X \in \mathcal P$.
	\end{enumerate}
\end{lemma}

\subsection{Polymatroids}
We outline basic facts and definitions of polymatroids. Polymatroids were introduced by Edmonds in \cite{Edmonds1970}. For background on submodular rank functions, see also
\cite{Lovasz1983}.

\begin{definition}
A \textbf{polymatroid} is a pair $P=(E,\rho)$, where $E$ is a finite set and the \textbf{rank function} $\rho:2^E\to\mathbb{R}$ satisfies:
\begin{enumerate}
    \item[($\rho$1)] $\rho(\emptyset)=0$,
    \item[($\rho$2)] $\rho(A)\leq\rho(B)$ whenever $A\subseteq B\subseteq E$,
    \item[($\rho$3)] $\rho(A\cup B)+\rho(A\cap B)\leq\rho(A)+\rho(B)$ for all $A,B\subseteq E$.
\end{enumerate}
\end{definition}

For a positive integer $r$, an \textbf{$r$-polymatroid} is a polymatroid with $\rho(\{x\})\leq r$ for all $x\in E$. An \textbf{integer polymatroid} has rank function with values in $\mathbb{Z}_{\geq0}$, and an integer $1$-polymatroid is a \textbf{matroid}. An element $x$ with $\rho(\{x\})=0$ is a \textbf{loop}. For $A\subseteq E$, the \textbf{restriction} $P|A$ is the polymatroid on $A$ with rank function $\rho|_{2^A}$, and the \textbf{contraction} $P/A$ is the polymatroid on $E\setminus A$ with rank function $X\mapsto\rho(X\cup A)-\rho(A)$. The \textbf{characteristic polynomial} of an integer polymatroid $P$ is defined as
$$p_P(z):=\sum\limits_{A\subseteq E}(-1)^{|A|}z^{\rho(E)-\rho(A)}.$$

We use the following matroid notions; see \cite{oxley2011matroid}.
Let $M=(E,r)$ be a matroid. A set $A\subseteq E$ is
\textbf{dependent} if $r(A)<|A|$, and a \textbf{circuit} is a
minimal dependent set. The \textbf{girth} $g(M)$ is the minimum
size of a circuit, with $g(M):=\infty$ if $M$ has no circuits.
The matroid $M$ is \textbf{simple} if it has no circuits of size
one or two, equivalently, if $g(M)\geq3$. A \textbf{flat} is a
set $F\subseteq E$ such that $r(F\cup\{x\})>r(F)$ for every
$x\in E\setminus F$. A matroid is \textbf{connected} if every
two distinct elements lie in a common circuit, and a flat $F$
is \textbf{connected} if the restriction $M|F$ is
connected. For a graph $\mG$ with vertex set $V$ and edge set $\mE$, the \textbf{cycle matroid}
$M(\mG)$ is the matroid on $\mE$ with
$r(A)=|V|-\kappa_{\mG}(A)$, where $\kappa_{\mG}(A)$ is the number of
connected components of the spanning subgraph $(V,A)$; see
\cite[Section~1.3]{oxley2011matroid}. We write $K_N$ for the
complete graph on the vertex set $[N]$ and, for $Y\subseteq[N]$, $\binom{Y}{2}$ for
the set of edges of $K_N$ with both ends in $Y$, that is, the edge set of the complete graph on $Y$.

In \cite{whittle1994critical}, Whittle uses the following notion of representability for polymatroids, which generalises representable matroids by allowing the elements of the ground set to be represented by subspaces of arbitrary dimension.

\begin{definition}\label{def:representable-polymatroid}
Given a field $\K$, an integer polymatroid $P=(E,\rho)$ is \textbf{$\K$-representable} if there are a finite-dimensional $\K$-vector space $W$ and subspaces $U_x\subseteq W$, $x\in E$, such that $\rho(A)=\dim_{\K}(\sum\limits_{x\in A}U_x)$ for every $A\subseteq E$. The family $(U_x)_{x\in E}$ is then a \textbf{$\K$-representation} of $P$.
\end{definition}

\subsection{The critical problem of projective geometries}

The \textbf{critical problem} of Crapo and Rota asks for the minimum number of hyperplanes needed to distinguish a given set of points of a finite projective space; see \cite[Chapter~16]{crapo1970foundations}. More precisely, let $S$ be a nonempty set of points of $\PG(k-1,q)$. Hyperplanes $\PG(\Lambda_1),\ldots,\PG(\Lambda_t)$ \textbf{distinguish} $S$ if each point of $S$ lies outside at least one of them, that is, if $$S\cap\PG(\Lambda_1)\cap\cdots\cap\PG(\Lambda_t)=\emptyset.$$
The \textbf{critical exponent} $\crit(S)$ of $S$ is the smallest $t$ for which such hyperplanes exist. For $k=1$, with the convention of the Notation above, the empty hyperplane distinguishes $S$, so $\crit(S)=1$ if $S\neq\emptyset$.

It is often convenient to work with linear forms rather than hyperplanes. We say that $\F_q$-linear forms $f_1,\ldots,f_t:\F_q^k\to\F_q$ \textbf{distinguish} $S$ if for every $\langle v\rangle\in S$ there is an $i$ with $f_i(v)\neq0$. Nonzero forms distinguish $S$ if and only if the hyperplanes $\PG(\ker f_1),\ldots,\PG(\ker f_t)$ do. Hence $\crit(S)$ is also the smallest number of linear forms distinguishing $S$. Since $\ker f_1\cap\cdots\cap\ker f_t$ has codimension at most $t$, $\crit(S)$ is also the minimum codimension of a subspace $W\subseteq\F_q^k$ such that $\PG(W)$ contains no point of $S$.

The critical problem depends only on the lattice $\mL(S)$ of subspaces of $\F_q^k$ spanned by subsets of $S$, where a set of points spans the vector subspace generated by their representatives \cite{crapo1970foundations}. Assume that $S$ spans $\F_q^k$, and let $\mu$ be the M\"obius function of~$\mL(S)$. The \textbf{characteristic polynomial} of $\mL(S)$ is
$$p_S(z)=\sum\limits_{x\in \mL(S)}\mu(0,x)z^{k-\dim x}.$$

Crapo and Rota provided a theoretical solution of the critical problem. Indeed, they proved that, for every integer $t\geq0$, the number of ordered $t$-tuples of linear forms on $\F_q^k$, zero forms included, that distinguish~$S$ is $p_S(q^t)$; see \cite[Chapter~16]{crapo1970foundations}. Hence $\crit(S)=\min\{t\geq0:p_S(q^t)>0\}$.

This has an interpretation in coding theory; see also \cite{dowling1971codes,greene1976weight}. Let $G$ be a $k\times n$ generator matrix of a linear $[n,k]_q$ code $\mC$ without zero columns, and let $S$ be the set of points represented by the columns of $G$. A linear form $f$ on $\F_q^k$ corresponds to the codeword $(f(g_1),\ldots,f(g_n))$, where $g_i$ are the columns of $G$, and $f_1,\ldots,f_t$ distinguish $S$ exactly when the supports of the corresponding codewords cover $[n]$. Therefore $\crit(S)$ is the smallest dimension of a subcode $\mD\subseteq\mC$ with $\supp(\mD)=[n]$.

\subsection{Additive codes and \texorpdfstring{$h$}{h}-projective systems}\label{subsec:additive}
We recall now additive codes and their correspondence with $h$-projective systems; see \cite{ball2025griesmer,ball2020additive,kurz2024additive}.

An $[n,k/h]_q^h$ \textbf{additive code} $\mC$ is a $k$-dimensional $\F_q$-linear subspace of~$\F_{q^h}^n$. The number $k/h=\log_{q^h}|\mC|$ is its \textbf{normalised dimension} and it need not be an integer (it plays the role of the dimension of $\mC$ over $\F_{q^h}$). For $v\in\F_{q^h}^n$ we write $\supp(v):=\{i\in[n]:v_i\neq0\}$ and $\wt(v):=|\supp(v)|$, and $\dd(u,v):=\wt(u-v)$ is the classical \textbf{Hamming distance}. The \textbf{minimum (Hamming) distance} of $\mC$ is $d(\mC):=\min\{\wt(c):c\in\mC\setminus\{0\}\}$. When $d=d(\mC)$ is known, we say that $\mC$ is an $[n,k/h,d]_q^h$ code.

A matrix $G\in\Mat_{k,n}(\F_{q^h})$ is a \textbf{generator matrix} of $\mC$ if its rows are $\F_q$-linearly independent and span $\mC$ over $\F_q$. For $i\in[n]$, the \textbf{$i$-th coordinate map} of $\mC$ is the $\F_q$-linear map $\varphi_i:\mC\to\F_{q^h}$, $c\mapsto c_i$. The code $\mC$ is \textbf{nondegenerate} if no $\varphi_i$ is zero, and \textbf{faithful} if every $\varphi_i$ is surjective; see~\cite{kurz2024additive}. Two additive codes are \textbf{equivalent} if one can be obtained from the other by permuting the coordinates and by applying an $\F_q$-linear bijection $\F_{q^h}\to\F_{q^h}$ in each coordinate.

Let $G\in\Mat_{k,n}(\F_{q^h})$ be a generator matrix of $\mC$ with columns $g_i=(g_{1,i},\ldots,g_{k,i})^\top\in\F_{q^h}^k$. Fix an ordered $\F_q$-basis $(\alpha_1,\ldots,\alpha_h)$ of $\F_{q^h}$ and write $g_{j,i}=\sum\limits_{t=1}^h b_{j,i}^t\alpha_t$ with $b_{j,i}^t\in\F_q$. This gives the \textbf{expanded generator matrix}
\begin{equation}\label{eq:Gtilde}
\widetilde{G}=(G_1\mid\cdots\mid G_n)\in\Mat_{k,nh}(\F_q),
\quad\text{where}\quad
G_i=\begin{pmatrix}
b^{1}_{1,i} & \cdots & b^{h}_{1,i}\\
\vdots & & \vdots\\
b^{1}_{k,i} & \cdots & b^{h}_{k,i}
\end{pmatrix}\in\Mat_{k,h}(\F_q).
\end{equation}
We refer to the matrices $G_i$ as the \textbf{blocks} of $\widetilde G$. The \textbf{expanded code}
$$\widetilde\mC:=\{a\widetilde G:a\in\F_q^k\}\subseteq\F_q^{hn}$$
is the image of $\mC$ under the coordinatewise expansion $\F_{q^h}\to\F_q^h$ with respect to $(\alpha_1,\ldots,\alpha_h)$. We regard $\F_q^{hn}$ as $(\F_q^h)^n$ and write $x=(x_1,\ldots,x_n)$ with $x_i\in\F_q^h$. Following \cite[Definition~1]{martinez2025folded}, the \textbf{folded Hamming weight} of $x$ is $\wt_{\rm F}(x):=|\{i\in[n]:x_i\neq0\}|$. If $x\in\widetilde\mC$ is the expansion of $c\in\mC$, then $x_i=0$ if and only if $c_i=0$, so $\wt_{\rm F}(x)=\wt(c)$.

The \textbf{dual} of $\mC$ is $\widetilde\mC^{\perp}:=\{x\in\F_q^{hn}:\widetilde Gx^\top=0\}$, the dual of $\widetilde\mC$ with respect to the standard inner product of $\F_q^{hn}$; see \cite[Definition~3]{martinez2025folded}. It has $\F_q$-dimension $hn-k$. We set $d^\perp(\mC):=\min\{\wt_{\rm F}(x):0\neq x\in\widetilde\mC^\perp\}$, with $d^\perp(\mC)=\infty$ if $\widetilde\mC^\perp=0$.
Another basis of $\F_{q^h}$ or another generator matrix changes $\widetilde\mC^\perp$ by invertible linear maps on the individual blocks $x_i$, which preserve folded Hamming weights, so $d^\perp(\mC)$ depends only on $\mC$, as shown in \cite[Proposition~5 and Corollary~7]{martinez2025folded}.

We now recall the geometric description of additive codes; see \cite{ball2020additive,ball2025griesmer}.
We set $U_i:=\colsp_{\F_q}(G_i)\subseteq\F_q^k$. For $a\in\F_q^k$, the $i$-th entry of the codeword $aG$ is $\sum\limits_t(aG_i)_t\alpha_t$, so it vanishes if and only if $aG_i=0$. Consequently $\dim U_i=\dim\im\varphi_i$. In particular, $\mC$ is nondegenerate if and only if every $U_i$ is nonzero, and faithful if and only if $\dim U_i=h$ for every~$i$.

For a nondegenerate code, $\pi_i:=\PG(U_i)$ is a projective subspace of $\PG(k-1,q)$ of projective dimension at most $h-1$, and we let $\mX_G(\mC)=\{\pi_1,\ldots,\pi_n\}$ be the resulting multiset associated with $G$. Another generator matrix changes this multiset by a projective linear transformation, so its equivalence class depends only on $\mC$. We write $\mX(\mC)$ for any representative. With the convention for $k=1$ fixed in the Notation, namely that the only hyperplane is the empty set $\PG(\{0\})$, all statements below also hold for $k=1$.

\begin{definition}[{\cite{ball2025griesmer}}]
Let $\mX=\{\pi_1,\ldots,\pi_n\}$ be a multiset of nonempty subspaces of $\PG(k-1,q)$, each of projective dimension at most $h-1$, not all contained in a common hyperplane. Then $\mX$ is an \textbf{$h$-$[n,k,d]_q$ projective system} if every hyperplane contains at most $n-d$ elements of $\mX$ and some hyperplane contains exactly $n-d$ elements. Two such systems are \textbf{equivalent} if they differ by an element of $\mathrm{PGL}_k(\F_q)$.
\end{definition}

\begin{theorem}[{\cite[Theorem~5]{ball2025griesmer}}]
There is a one-to-one correspondence between equivalence classes of nondegenerate $[n,k/h,d]_q^h$ codes and equivalence classes of $h$-$[n,k,d]_q$ projective systems.
\end{theorem}

The correspondence can be explained as follows. For $a\in\F_q^k$ let
$$\Lambda_a:=\langle a\rangle^\perp=\{x\in\F_q^k:ax^\top=0\}\subseteq \F_q^{k},$$
which is a vector hyperplane if $a\neq0$. If $c=aG$ with $a\neq0$, then $c_i=0$ if and only if $\pi_i\subseteq\PG(\Lambda_a)$, so $\wt(c)=n-|\{i:\pi_i\subseteq\PG(\Lambda_a)\}|$. Conversely, from an $h$-$[n,k,d]_q$ projective system with $\pi_i=\PG(U_i)$, choose $G_i\in\Mat_{k,h}(\F_q)$ with $\colsp_{\F_q}(G_i)=U_i$, adding zero columns if $\dim U_i<h$. Then $\rowsp_{\F_q}(G_1\boldsymbol{\alpha}\mid\cdots\mid G_n\boldsymbol{\alpha})$, with $\boldsymbol{\alpha}=(\alpha_1,\ldots,\alpha_h)^\top$, is an $[n,k/h,d]_q^h$ code.

Every $[n,k/h,d]_q^h$ additive code satisfies the \textbf{Singleton bound} \cite{ball2025griesmer}:
\begin{equation}\label{eq:singleton}
d\leq n-\lceil k/h\rceil+1.
\end{equation}
 A code attaining it is called \textbf{quasi-MDS}, or \textbf{QMDS} according to \cite[Definition~9]{martinez2025folded}; see also \cite{bartoli2025long}. A QMDS code with $h\mid k$ is an \textbf{MDS} code, that is, $k=h(n-d+1)$. When $h\nmid k$, QMDS codes are the fractional MDS codes of \cite{ball2025griesmer}. Since $\widetilde\mC^\perp$ has $\F_q$-dimension $hn-k$, the Singleton bound for the folded Hamming weight applied to the dual gives
\begin{equation}\label{eq:dual-singleton}
d^\perp(\mC)\leq\lfloor k/h\rfloor+1
\end{equation}
whenever $\widetilde\mC^\perp\neq0$; see \cite[Proposition~8]{martinez2025folded}.

\subsection{Hypergraphs}\label{subsec:hypergraphs}

We use the following notions; see \cite{berge1989hypergraphs,bretto2013hypergraph} for hypergraphs and \cite{brouwer2012spectra} for spectra of graphs. A finite \textbf{hypergraph} is a triple $H=(V,\mE,\psi_H)$, where $V$ and $\mE$ are finite sets and $\psi_H:\mE\to2^V\setminus\{\emptyset\}$. The elements of $V$ are \textbf{vertices}, the elements of $\mE$ are \textbf{edges}, and $\psi_H(e)$ is the set of vertices of $e$. We write $N:=|V|$ and $m:=|\mE|$. When no confusion can arise we identify $e$ with $\psi_H(e)$. Note that distinct edges may have the same vertex set. The hypergraph $H$ is \textbf{$r$-uniform} if $|e|=r$ for every edge, \textbf{simple} if $\psi_H$ is injective, and \textbf{linear} if any two distinct edges have at most one common vertex. A $2$-uniform hypergraph is a graph (possibly with parallel edges). We write $K_N^{(r)}$ for the \textbf{complete $r$-uniform hypergraph}, whose edges are all $r$-subsets of $[N]$. The \textbf{degree} $\deg(v)$ of a vertex $v$ is the number of edges containing it. Let $\Delta(H)$ and $\delta_{\min}(H)$ denote the maximum and minimum degree, respectively. We say that $H$ is \textbf{regular of degree $\Delta$} if every vertex has degree $\Delta$.

\medskip\noindent\emph{Connectivity.} A \textbf{walk} from $u$ to $v$ is a sequence $u_0,e_1,u_1,\ldots,e_\ell,u_\ell$ with $u_0=u$, $u_\ell=v$ and $u_{i-1},u_i\in e_i$ for every $i$. Being joined by a walk is an equivalence relation on $V$, whose classes are the \textbf{connected components} of $H$, and $H$ is \textbf{connected} if it has one component \cite{berge1989hypergraphs}. For $A\subseteq \mE$, the \textbf{spanning subhypergraph} $H_A:=(V,A,\psi_H|_A)$ has $\kappa(A)$ components, isolated vertices included, and $\kappa(H):=\kappa(\mE)$. We write
$$V(A):=\bigcup_{e\in A}e.$$
For $\emptyset\neq S\subsetneq V$, the \textbf{edge cut} $\delta_H(S)$ is the set of edges meeting both $S$ and $V\setminus S$. The \textbf{edge-connectivity} $\lambda(H)$ of a connected hypergraph with $N\geq2$ is the minimum of $|\delta_H(S)|$ over all $\emptyset\neq S\subsetneq V$. The \textbf{incidence graph} of $H$ is the bipartite graph with vertex set $V\sqcup \mE$ in which $v$ is adjacent to $e$ when $v\in e$.

\medskip\noindent\emph{Berge cycles.} A \textbf{Berge cycle} of length $\ell\geq2$ is a sequence $e_1,v_1,e_2,v_2,\ldots,e_\ell,v_\ell$ of distinct edges $e_i$ and distinct vertices $v_i$ such that $v_i\in e_i\cap e_{i+1}$ for every $i$, where $e_{\ell+1}:=e_1$ \cite{berge1989hypergraphs}. Equivalently, it is a cycle of length $2\ell$ in the incidence graph. The \textbf{Berge girth} $g_{\rm B}(H)$ is the minimum length of a Berge cycle, with $g_{\rm B}(H)=\infty$ if there is none; see \cite{ellis2014regular}. Two distinct edges with two common vertices form a Berge cycle of length $2$, so $g_{\rm B}(H)\geq3$ if and only if $H$ is linear. For graphs, the Berge girth is the usual girth.

\medskip\noindent\emph{Colourings.} A \textbf{weak colouring} of $H$ with $s$ colours is a map $\gamma:V\to[s]$ such that no edge is \textbf{monochromatic}, that is, $\gamma$ is not constant on any edge; see \cite{erdos1966chromatic,berge1989hypergraphs}. We denote by $\chi_H^{\mathrm w}(s)$ the number of weak colourings with $s$ colours and by $\chi_{\mathrm w}(H)$ the \textbf{weak chromatic number}, that is, the least $s$ for which one exists. It is finite if and only if every edge has at least two vertices. Helgason \cite{helgason1974} showed that $\chi_H^{\mathrm w}$ is a polynomial, the \textbf{weak chromatic polynomial}; see also \cite{whittle1992hypergraph,dohmen1995broken,tomescu1998chromatic}. A \textbf{strong colouring} requires instead that any two vertices of a common edge receive different colours \cite{berge1989hypergraphs}. For graphs both notions agree with proper colourings. The \textbf{weak independence number} $\alpha_{\mathrm w}(H)$ is the maximum size of a set of vertices containing no edge.

\medskip\noindent\emph{The weighted $2$-section and its Laplacian.} The \textbf{codegree} of two distinct vertices $u,v$ is $a_{uv}:=|\{e\in \mE:u,v\in e\}|$. The \textbf{$2$-section} of $H$ is the graph on $V$ in which two distinct vertices are adjacent if they lie in a common edge \cite{berge1989hypergraphs,bretto2013hypergraph}. In the \textbf{weighted $2$-section}, the edge $\{u,v\}$ has weight $a_{uv}$. Its adjacency matrix is the codegree matrix $A_H=(a_{uv})$, with $a_{uu}:=0$. Let $D_H$ be the diagonal matrix with entries $\sum\limits_{v\neq u}a_{uv}$. The \textbf{Laplacian} of the weighted $2$-section is $L_H:=D_H-A_H$; see \cite{rodriguez2002laplacian,rodriguez2009laplacian}. It is symmetric and positive semidefinite, since $x^\top L_Hx=\sum\limits_{\{u,v\}}a_{uv}(x_u-x_v)^2$, and $L_H\mathbf 1=0$. Its \textbf{Laplacian spectrum} is the multiset of its eigenvalues $0=\mu_1\leq\mu_2\leq\cdots\leq\mu_N$. The second smallest eigenvalue $\mu_2$ is positive if and only if the weighted $2$-section is connected \cite{fiedler1973algebraic,brouwer2012spectra}, which is the case if and only if $H$ is connected.

\begin{example}[The Fano plane]\label{ex:fano}
Let $V=[7]$ and let $\mE$ consist of the seven lines of the Fano plane,
$$
\{1,2,3\},\ \{1,4,5\},\ \{1,6,7\},\ \{2,4,6\},\ \{2,5,7\},\ \{3,4,7\},\ \{3,5,6\}.
$$
This hypergraph $H$ is $3$-uniform, simple, linear and regular of degree $3$. Any two lines meet in one point, and $\{1,2,3\},2,\{2,4,6\},4,\{1,4,5\},1$ is a Berge cycle of length $3$, so $g_{\rm B}(H)=3$. Every pair of points lies on exactly one line, so the weighted $2$-section is $K_7$ with all weights equal to $1$, $L_H=7I-J$, and the Laplacian spectrum is $\{0,7,7,7,7,7,7\}$. Isolating a vertex cuts the three lines through it, so $\lambda(H)\leq3$. Conversely, let $\emptyset\neq S\subsetneq V$. Each of the $|S|(7-|S|)\geq6$ pairs $\{u,v\}$ with $u\in S$ and $v\notin S$ lies on exactly one line, and a line contains at most two such pairs, so $|\delta_H(S)|\geq3$. Hence $\lambda(H)=3$. Enumerating all colourings shows that $H$ has no weak colouring with two colours but has weak colourings with three colours, so $\chi_{\mathrm w}(H)=3$. More precisely,
$$
\chi_H^{\mathrm w}(z)=z(z-1)(z-2)(z^4+3z^3-6z+3).
$$
\end{example}

\begin{example}\label{ex:small-hypergraph}
Let $H_0$ be the $3$-uniform hypergraph on $V=[4]$ with edges $e_1=\{1,2,3\}$, $e_2=\{1,2,4\}$ and $e_3=\{1,3,4\}$. It is connected and not linear, and $e_1,1,e_2,2$ is a Berge cycle of length $2$, so $g_{\rm B}(H_0)=2$. Vertex $1$ has degree $3$ and the other vertices have degree $2$. Isolating vertex $2$ cuts $e_1,e_2$, and one checks that $\lambda(H_0)=2$. The codegrees are $a_{12}=a_{13}=a_{14}=2$ and $a_{23}=a_{24}=a_{34}=1$, and the Laplacian spectrum of the weighted $2$-section is $\{0,5,5,8\}$.
\end{example}

\section{Additive codes and the critical problem}\label{sec:polymatroid}

In this section we associate a representable polymatroid with an additive code, express the minimum distance and the dual distance through it, prove the Critical Theorem in this setting, and give bounds on the critical exponent.

\subsection{The polymatroid of an additive code}

For $A\subseteq[n]$ let
$$\F_{q^h}^n(A):=\{v\in\F_{q^h}^n:\supp(v)\subseteq[n]\setminus A\}$$
be the $\F_q$-subspace of vectors vanishing on $A$.

\begin{definition}
Let $\mC$ be an $[n,k/h]_q^h$ additive code and $A\subseteq[n]$. The \textbf{shortening} of $\mC$ by $A$ is $\mC(A):=\mC\cap\F_{q^h}^n(A)$, the $\F_q$-subspace of codewords vanishing on $A$. Unlike the usual shortening operation, we retain the zero coordinates indexed by $A$. We define $\rho_{\mC}:2^{[n]}\to\mathbb Z_{\geq0}$ by $\rho_{\mC}(A):=k-\dim_{\F_q}\mC(A)$.
\end{definition}

Clearly $\mC(\emptyset)=\mC$, $\mC([n])=0$, and $\mC(A)\supseteq\mC(B)$ whenever $A\subseteq B$.

\begin{theorem}
Let $\mC$ be an $[n,k/h]_q^h$ additive code. Then $P_{\mC}:=([n],\rho_{\mC})$ is an integer $h$-polymatroid with $\rho_{\mC}([n])=k$ and $\rho_{\mC}(\{i\})=\dim_{\F_q}\im\varphi_i$. In particular, $P_{\mC}$ has no loops if and only if $\mC$ is nondegenerate, and $\rho_{\mC}(\{i\})=h$ for all $i$ if and only if $\mC$ is faithful.
\end{theorem}

\begin{proof}
We have $\rho_{\mC}(\emptyset)=0$ and $\rho_{\mC}([n])=k$, and monotonicity follows from $\mC(A)\supseteq\mC(B)$ for $A\subseteq B$. For submodularity, let $A,B\subseteq[n]$. A codeword lies in $\mC(A)\cap\mC(B)$ if and only if it vanishes on $A\cup B$, so $\mC(A)\cap\mC(B)=\mC(A\cup B)$. Moreover $\mC(A)+\mC(B)\subseteq\mC(A\cap B)$. By Grassmann's identity,
$$
\dim\mC(A)+\dim\mC(B)=\dim(\mC(A)+\mC(B))+\dim\mC(A\cup B)\leq\dim\mC(A\cap B)+\dim\mC(A\cup B),
$$
which is equivalent to $\rho_{\mC}(A\cup B)+\rho_{\mC}(A\cap B)\leq\rho_{\mC}(A)+\rho_{\mC}(B)$. Finally, the kernel of the coordinate map $\varphi_i$ is $\mC(\{i\})$, so by the rank-nullity theorem $\rho_{\mC}(\{i\})=\dim_{\F_q}\im\varphi_i\leq h$.
\end{proof}

Let $G$ be a generator matrix of $\mC$, let $\widetilde G=(G_1\mid\cdots\mid G_n)$ be its expansion as in \eqref{eq:Gtilde}, and let $U_i=\colsp_{\F_q}(G_i)\subseteq\F_q^k$.

\begin{proposition}\label{prop:representable-code-polymatroid}
For every $A\subseteq[n]$ we have $\rho_{\mC}(A)=\dim_{\F_q}\left(\sum\limits_{i\in A}U_i\right)$. In particular, $P_{\mC}$ is represented by $(U_i)_{i\in[n]}$.
\end{proposition}

\begin{proof}
The map $\F_q^k\to\mC$, $a\mapsto aG$, is an $\F_q$-linear isomorphism, and $aG\in\mC(A)$ if and only if $aG_i=0$ for every $i\in A$, that is, if and only if $a\in\left(\sum\limits_{i\in A}U_i\right)^\perp$. Hence $\dim_{\F_q}\mC(A)=k-\dim_{\F_q}\sum\limits_{i\in A}U_i$.
\end{proof}

\begin{example}\label{ex:small-code}
Let $\F_4=\F_2(\omega)$ with $\omega^2=\omega+1$, and let $\mC\subseteq\F_4^3$ be the $\F_2$-row space of
$$
G=\begin{pmatrix}1&0&1\\ \omega&1&\omega\\ 0&\omega&1\end{pmatrix}.
$$
With respect to the basis $(1,\omega)$, the blocks of $\widetilde G$ have column spaces $U_1=\langle\varepsilon_1,\varepsilon_2\rangle$, $U_2=\langle\varepsilon_2,\varepsilon_3\rangle$ and $U_3=\langle\varepsilon_1+\varepsilon_3,\varepsilon_2\rangle$ in $\F_2^3$, where $(\varepsilon_1,\varepsilon_2,\varepsilon_3)$ is the standard basis. Thus $\mC$ is a faithful $[3,3/2]_2^2$ code, and $\rho_{\mC}(A)=2$ if $|A|=1$ and $\rho_{\mC}(A)=3$ if $|A|\geq2$. The seven nonzero codewords are
$$
(1,0,1),\ (\omega,1,\omega),\ (0,\omega,1),\ (\omega^2,1,\omega^2),\ (1,\omega,0),\ (\omega,\omega^2,\omega^2),\ (\omega^2,\omega^2,\omega),
$$
with weights $2,3,2,3,2,3,3$, so $d(\mC)=2$. The code is not $\F_4$-linear, since $\omega\cdot(1,0,1)\notin\mC$.
\end{example}

The minimum distance and the dual distance of $\mC$ are determined by $P_{\mC}$.

\begin{lemma}\label{lem:distance-rank}
Let $\mC\subseteq\F_{q^h}^n$ be a nonzero additive code of $\F_q$-dimension $k$. Then
$$
d(\mC)=n-\max\{|A|:A\subseteq[n],\ \rho_{\mC}(A)<k\}.
$$
\end{lemma}

\begin{proof}
We have $\rho_{\mC}(A)<k$ if and only if $\mC(A)\neq0$, that is, if and only if some nonzero codeword vanishes on $A$. Hence the maximum is equal to  $\max\limits_{c\neq0}(n-\wt(c))=n-d(\mC)$.
\end{proof}

\begin{lemma}\label{lem:block-girth-rank}
For every additive code $\mC$,
$$
d^\perp(\mC)=\min\{|A|:A\subseteq[n],\ \rho_{\mC}(A)<h|A|\},
$$
where $\min\emptyset=\infty$. In particular, if $\mC$ is not faithful, $d^\perp(\mC)=1$.
\end{lemma}

\begin{proof}
Let $G_A$ be the submatrix of $\widetilde G$ formed by the blocks indexed by $A$. By \cref{prop:representable-code-polymatroid}, $\rank(G_A)=\rho_{\mC}(A)$. A nonzero $x\in\widetilde\mC^\perp$ with $x_i=0$ for all $i\notin A$ is the same as a nonzero vector in the kernel of $G_A$, and since $G_A$ has $h|A|$ columns, such a vector exists if and only if $\rho_{\mC}(A)<h|A|$.
\end{proof}

The two lemmas describe dual extremal problems on $P_{\mC}$: $n-d(\mC)$ is the largest size of a set $A$ with $\rho_{\mC}(A)<\rho_{\mC}([n])$, and $d^\perp(\mC)$ is the smallest size of a set $A$ with $\rho_{\mC}(A)<h|A|$. For $h=1$, where $P_{\mC}$ is the matroid of $\mC$, these are the classical facts that $n-d$ is the largest size of a hyperplane of the matroid and that $d^\perp$ is its girth. For the code of \cref{ex:small-code}, singletons have rank $2<3$ and $\rho_{\mC}(\{1,2\})=3<4$, so $d(\mC)=3-1=2$ and $d^\perp(\mC)=2$.

\subsection{The critical exponent}\label{sec:critical}

Throughout this subsection and the next, $\mC$ is a nondegenerate $[n,k/h]_q^h$ additive code and $P_{\mC}=([n],\rho_{\mC})$ is its polymatroid. The characteristic polynomial of $P_{\mC}$ is then
$$
p_{\mC}(z)=\sum\limits_{A\subseteq[n]}(-1)^{|A|}z^{k-\rho_{\mC}(A)}=\sum\limits_{A\subseteq[n]}(-1)^{|A|}z^{\dim_{\F_q}\mC(A)}.
$$

\begin{definition}\label{def:critical-code}
The \textbf{support} of an $\F_q$-subspace $\mD\subseteq\mC$ is $\supp(\mD):=\bigcup\limits_{c\in\mD}\supp(c)$. The \textbf{critical exponent} of $\mC$ is
$$
\crit(\mC):=\min\{\dim_{\F_q}\mD:\mD\subseteq\mC\text{ an $\F_q$-subspace with }\supp(\mD)=[n]\}.
$$
\end{definition}

Since the support of a subcode is the union of the supports of any set of generators, $\crit(\mC)$ is the smallest $t$ such that some codewords $c_1,\ldots,c_t$ have supports covering $[n]$. Equivalently, $\crit(\mC)$ is the least $f\geq1$ such that some $f$-dimensional subcode has full support. This should not be confused with the generalised Hamming weights of \cite[(5)]{d2026generalized}, which are defined by the \emph{smallest} support of an $f$-dimensional subcode.

We translate this into the language of projective systems. Let $G$, $G_i$, $U_i$, $\pi_i=\PG(U_i)$ and $\Lambda_a$ be as in \cref{subsec:additive}. For $c=aG$ we have $c_i=0$ if and only if $U_i\subseteq\Lambda_a$.

\begin{definition}\label{def:critical-projective-system}
Let $\mX=\{\pi_1,\ldots,\pi_n\}$ be a multiset of nonempty subspaces of $\PG(k-1,q)$. Hyperplanes $\PG(\Lambda_1),\ldots,\PG(\Lambda_t)$ \textbf{distinguish} $\mX$ if for every $i\in[n]$ some $\PG(\Lambda_j)$ does not contain $\pi_i$. The \textbf{critical exponent} $\crit(\mX)$ is the minimum $t\geq1$ for which $t$ distinguishing hyperplanes exist.
\end{definition}

For $k=1$ the empty hyperplane distinguishes $\mX$, and $\crit(\mX)=1$.

\begin{proposition}\label{prop:critical-geometric}
If $\mC$ is nondegenerate, then $\crit(\mC)=\crit(\mX(\mC))$.
\end{proposition}

\begin{proof}
Let $c_1,\ldots,c_t\in\mC$. Zero codewords do not contribute to the support, so we may assume that $c_j=a_jG$ with $a_j\neq0$ for $j\in[t]$, and then every $\PG(\Lambda_{a_j})$ is a hyperplane. Coordinate $i$ lies in the support of some $c_j$ if and only if some $\PG(\Lambda_{a_j})$ does not contain $\pi_i$. Hence the supports of $c_1,\ldots,c_t$ cover $[n]$ if and only if $\PG(\Lambda_{a_1}),\ldots,\PG(\Lambda_{a_t})$ distinguish $\mX(\mC)$.
\end{proof}

By \cref{prop:representable-code-polymatroid}, $P_{\mC}$ is $\F_q$-representable, so we can apply Whittle's Critical Theorem \cite[Theorem~3.1]{whittle1994critical}. We give a direct proof in our setting. For $S\subseteq[n]$ we write $P_{\mC}.S:=P_{\mC}/([n]\setminus S)$ for the contraction of $P_{\mC}$ to $S$.

\begin{theorem}[Critical Theorem]\label{thm:critical-additive}
Let $\mC$ be an $[n,k/h]_q^h$ additive code, let $t\geq1$ and $S\subseteq[n]$. Then
$$
\bigl|\{(c_1,\ldots,c_t)\in\mC^t:\supp(\langle c_1,\ldots,c_t\rangle_{\F_q})=S\}\bigr|=p_{P_{\mC}.S}(q^t).
$$
\end{theorem}

\begin{proof}
Let $A:=[n]\setminus S$. For $B\supseteq A$, let $f(B)$ be the number of $t$-tuples in $\mC^t$ whose common zero set is exactly $B$, and $g(B)$ the number of those whose common zero set contains $B$. Then $g(B)=\sum\limits_{D\supseteq B}f(D)$, and $g(B)=|\mC(B)|^t=q^{t(k-\rho_{\mC}(B))}$, because a tuple has common zero set containing $B$ exactly when all its entries lie in $\mC(B)$. M\"obius inversion on the interval $[A,[n]]$ of the Boolean lattice gives
$$
f(A)=\sum\limits_{B\supseteq A}(-1)^{|B|-|A|}q^{t(k-\rho_{\mC}(B))}.
$$
The contraction $P_{\mC}/A$ has rank function $X\mapsto\rho_{\mC}(X\cup A)-\rho_{\mC}(A)$ and total rank $k-\rho_{\mC}(A)$, so the right-hand side is its characteristic polynomial evaluated at $q^t$. Finally, the common zero set of $(c_1,\ldots,c_t)$ is $A$ if and only if $\supp(\langle c_1,\ldots,c_t\rangle)=S$.
\end{proof}

\begin{corollary}\label{cor:critical-full-support}
If $\mC$ is nondegenerate, then for every $t\geq1$ the number of ordered $t$-tuples in $\mC^t$ whose supports cover $[n]$ is $p_{\mC}(q^t)$. Consequently, $\crit(\mC)=\min\{t\geq1:p_{\mC}(q^t)>0\}$.
\end{corollary}

\begin{proof}
Apply \cref{thm:critical-additive} with $S=[n]$, for which $P_{\mC}.[n]=P_{\mC}$.
\end{proof}

\begin{example}\label{ex:small-code-crit}
For the code of \cref{ex:small-code}, $p_{\mC}(z)=z^3-3z+2=(z-1)^2(z+2)$. Thus $p_{\mC}(2)=4$, which is the number of codewords of weight $3$, and $\crit(\mC)=1$.
\end{example}

The geometric formulation has an equivalent form in terms of a single subspace.

\begin{proposition}\label{prop:crit-codimension}
We have
$$
\crit(\mC)=\min\{\codim W:W\subseteq\F_q^k\text{ a subspace with }U_i\not\subseteq W\text{ for every }i\in[n]\}.
$$
\end{proposition}

\begin{proof}
For a subspace $W\subseteq\F_q^k$ we have $\pi_i\subseteq\PG(W)$ if and only if $U_i\subseteq W$. If the hyperplanes $\PG(\Lambda_1),\ldots,\PG(\Lambda_t)$ distinguish $\mX(\mC)$, then no $\pi_i$ is contained in $\PG(\Lambda_1\cap\cdots\cap\Lambda_t)$, so $W:=\Lambda_1\cap\cdots\cap\Lambda_t$ contains no $U_i$, and $\codim W\leq t$. Conversely, let $W$ have codimension $s$ and contain no $U_i$. Since every $U_i$ is nonzero, $W\neq\F_q^k$, so $s\geq1$. Write $W=\Lambda_1\cap\cdots\cap\Lambda_s$ with vector hyperplanes $\Lambda_j$, for instance the kernels of a basis of the annihilator of $W$. No $U_i$ is contained in all the $\Lambda_j$, so for every $i$ some $\PG(\Lambda_j)$ does not contain $\pi_i$. Hence $\PG(\Lambda_1),\ldots,\PG(\Lambda_s)$ distinguish $\mX(\mC)$, and $\crit(\mX(\mC))\leq s$. The claim follows from \cref{prop:critical-geometric}.
\end{proof}

Following \cite[Definition~1.1]{d2026generalized}, say that $\mX(\mC)$ \textbf{blocks} a subspace $W\subseteq\F_q^k$ if some $U_i$ is contained in $W$. Every subspace of an unblocked subspace is also unblocked. Thus, if the minimum codimension of an unblocked subspace is $t$, every subspace of codimension less than $t$ is blocked. By \cref{prop:crit-codimension}, this gives
$$
\crit(\mC)=1+\max\{f\geq0:\mX(\mC)\text{ blocks every subspace of codimension $f$}\}.
$$

For $i\in[n]$ let $\rho_i:=\rho_{\mC}(\{i\})=\dim_{\F_q}\im\varphi_i$. Since $\mC$ is nondegenerate, we have $1\leq\rho_i\leq h$. Let $\rho_{\min}:=\min\limits_i\rho_i$ and $w_{\max}(\mC):=\max\limits_{c\in\mC}\wt(c)$.

\begin{proposition}\label{prop:crit-general-bounds}
Let $\mC$ be a nondegenerate additive code of $\F_q$-dimension $k$.
Let $t_0:=\min\{t\geq1:\sum\limits_{i=1}^nq^{-t\rho_i}<1\}$.
Then
$$
\begin{aligned}
\left\lceil\frac{n}{w_{\max}(\mC)}\right\rceil
\leq\crit(\mC)
&\leq\min\{k-\rho_{\min}+1,t_0\}\\
&\leq\min\left\{k-\rho_{\min}+1,
\left\lfloor\frac{\log_q n}{\rho_{\min}}\right\rfloor+1\right\}.
\end{aligned}
$$
\end{proposition}

\begin{proof}
If $c_1,\ldots,c_t$ is a basis of a subcode with full support,
then
$n\leq\sum\limits_{j=1}^t\wt(c_j)\leq t\,w_{\max}(\mC)$.
This gives the lower bound.
Every $U_i$ has dimension $\rho_i\geq\rho_{\min}$.
A subspace $W\subseteq\F_q^k$ of dimension
$\rho_{\min}-1$ therefore contains none of the $U_i$.
By \cref{prop:crit-codimension},
$\crit(\mC)\leq\codim W=k-\rho_{\min}+1$.
For the other upper bound, fix $t\geq1$. A tuple
$(c_1,\ldots,c_t)\in\mC^t$ fails to have full joint
support only if all its entries vanish at some coordinate
$i$. Since $|\mC(\{i\})|=q^{k-\rho_i}$, the number of
such tuples is at most
$$
\sum\limits_{i=1}^n|\mC(\{i\})|^t
=q^{tk}\sum\limits_{i=1}^nq^{-t\rho_i}.
$$
If $\sum\limits_{i=1}^nq^{-t\rho_i}<1$, this is smaller than
$|\mC^t|=q^{tk}$. Hence some $t$-tuple has full joint
support, and its span has dimension at most $t$.
Finally,
$\sum\limits_{i=1}^nq^{-t\rho_i}\leq nq^{-t\rho_{\min}}<1$
whenever $t>\log_q(n)/\rho_{\min}$.
\end{proof}

\begin{remark}\label{rem:wmax}
The lower bound of \cref{prop:crit-general-bounds} only detects whether $\mC$ has a codeword of full weight. Indeed, a uniformly random codeword $c\in\mC$ satisfies $c_i\neq0$ with probability $1-|\mC(\{i\})|/|\mC|=1-q^{-\rho_i}\geq1-1/q$, so the average weight of a codeword is at least $n(1-1/q)\geq n/2$. Hence $w_{\max}(\mC)\geq n/2$, and $\lceil n/w_{\max}(\mC)\rceil=\min\{\crit(\mC),2\}$. Better lower bounds for hypergraphic codes are given in \cref{cor:crit-hypergraph-bounds}.
\end{remark}

\begin{example}\label{ex:crit-general-sharp}
The bound $\crit(\mC)\leq k-h+1$ for faithful codes is sharp.
Let $k\geq h$, and take as coordinate subspaces all the
$h$-dimensional subspaces of $\F_q^k$. Every subspace of
dimension at least $h$ contains one of these coordinate
subspaces, whereas a subspace of dimension $h-1$ contains
none. Hence \cref{prop:crit-codimension} gives
$\crit(\mC)=k-h+1$.
\end{example}

\subsection{A Kung-type bound}\label{subsec:kung-bound}

In \cite{kung1996critical}, Kung bounded the critical exponent of linear codes in terms of the girth of the associated matroid. Let $M$ be a simple matroid of rank $k$ represented over $\F_q$ by a set $S$ of points of $\PG(k-1,q)$. By the Critical Theorem of Crapo and Rota, $\crit(S)$ depends only on $M$ and $q$, and we denote it by $\crit(M;q)$. Kung proved that $\crit(M;q)\leq k-g(M)+3$ in \cite{kung1996critical}. We extend this bound to additive codes. The girth is replaced by the dual distance $d^\perp(\mC)$. By \cref{lem:block-girth-rank}, it is the smallest size of a set $A$ of blocks for which the concatenated matrix $G_A$ has rank less than $h|A|$.

\begin{theorem}[Kung-type bound]\label{thm:kung-block}
Let $\mC\subseteq\F_{q^h}^n$ be a faithful additive code of $\F_q$-dimension~$k$, and let $d^\perp:=d^\perp(\mC)$. If $3\leq d^\perp<\infty$, then $\crit(\mC)\leq k-h(d^\perp-1)+2$.
\end{theorem}

\begin{proof}
By \cref{prop:crit-codimension}, it suffices to find a subspace of $\F_q^k$ of dimension $h(d^\perp-1)-2$ that contains none of $U_1,\ldots,U_n$. Since $d^\perp<\infty$, \cref{lem:block-girth-rank} gives $n\geq d^\perp$. Choose $B\subseteq[n]$ with $|B|=d^\perp-2$, which is nonempty since $d^\perp\geq3$, and choose $j_0\in[n]\setminus B$. For every $j\notin B$ the set $B\cup\{j\}$ has $d^\perp-1<d^\perp$ elements, so \cref{lem:block-girth-rank,prop:representable-code-polymatroid} give
$$
\dim\Bigl(U_j+\sum\limits_{i\in B}U_i\Bigr)=\rho_{\mC}(B\cup\{j\})=h(d^\perp-1)=\dim U_j+\sum\limits_{i\in B}\dim U_i.
$$
Hence the sum $W:=\sum\limits_{i\in B}U_i=\bigoplus\limits_{i\in B}U_i$ is direct, $\dim W=h(d^\perp-2)$, and $U_j\cap W=0$ for every $j\notin B$. Choose a nonzero linear form $\ell_i$ on $U_i$ for each $i\in B$. Since the sum is direct, there is a unique linear form $\ell$ on $W$ whose restriction to $U_i$ is $\ell_i$ for every $i\in B$. As $B\neq\emptyset$, $\ell\neq0$, so $W_0:=\ker\ell$ has dimension $h(d^\perp-2)-1$. Let $T$ be a hyperplane of $U_{j_0}$, and set $W_1:=W_0+T$. Since $T\subseteq U_{j_0}$ and $U_{j_0}\cap W=0$, the sum $W_0+T$ is direct, so $\dim W_1=h(d^\perp-1)-2$, and $W_1\cap W=W_0+(T\cap W)=W_0$.

Let $i\in B$. If $U_i\subseteq W_1$, then $U_i\subseteq W_1\cap W=W_0=\ker\ell$, which contradicts $\ell|_{U_i}=\ell_i\neq0$. Let $j\notin B$. Then $U_j\cap W_0\subseteq U_j\cap W=0$, so the projection $W_1\to W_1/W_0$ is injective on $U_j\cap W_1$. Since $\dim W_1/W_0=h-1<\dim U_j$, we get $U_j\not\subseteq W_1$. Hence $W_1$ contains none of $U_1,\ldots,U_n$, and \cref{prop:crit-codimension} gives $\crit(\mC)\leq\codim W_1=k-h(d^\perp-1)+2$.
\end{proof}

For a faithful code, $\rho_{\min}=h$, so \cref{prop:crit-general-bounds} gives $\crit(\mC)\leq k-h+1$.
If $d^\perp(\mC)\geq3$, \cref{thm:kung-block} improves this by $h(d^\perp(\mC)-2)-1$, which is positive
unless $h=1$ and $d^\perp(\mC)=3$. By \cref{lem:block-girth-rank}, any $d^\perp(\mC)-1$ blocks have rank
$h(d^\perp(\mC)-1)$, so $k\geq h(d^\perp(\mC)-1)$ and the bound of \cref{thm:kung-block} is always at
least~$2$. It is attained for every $h$ and $q$ by the codes of spreads (see \cref{ex:spread} below), and for
$h\geq2$ and $d^\perp(\mC)=3$ these are the only codes attaining it.

\begin{example}\label{ex:spread}
Let $\mathcal S$ be a spread of $\F_q^{2h}$, that is, a set of $q^h+1$ subspaces of dimension $h$ that pairwise intersect trivially, for instance a Desarguesian spread. Let $\mC$ be the faithful $[q^h+1,2,q^h]_q^h$ code with $\mX(\mC)=\{\PG(U):U\in\mathcal S\}$; for a Desarguesian spread, $\mC$ is the $\F_{q^h}$-linear $[q^h+1,2,q^h]_{q^h}$ code whose columns represent the points of $\PG(1,q^h)$. Any two elements of $\mathcal S$ span $\F_q^{2h}$, so \cref{lem:block-girth-rank} gives $d^\perp(\mC)=3$. A hyperplane $\Lambda$ of $\F_q^{2h}$ meets every $U\in\mathcal S$ in dimension $h$ or $h-1$. Since the elements of $\mathcal S$ partition the nonzero vectors, counting the points of $\PG(\Lambda)$ shows that exactly one element of $\mathcal S$ is contained in $\Lambda$. Hence every hyperplane is blocked, so $\crit(\mC)\geq2$ by \cref{prop:crit-codimension}, and \cref{thm:kung-block} gives $\crit(\mC)\leq2h-2h+2=2$. Thus $\crit(\mC)=2$, whereas $k-h+1=h+1$.
\end{example}
 
\begin{remark}\label{rem:kung-h1}
Let $h=1$, so that $\mC$ is a linear $[n,k]_q$ code without zero coordinates, and let $M$ be the matroid on $[n]$ represented by the columns of a generator matrix $G$. Then $P_{\mC}=M$ by \cref{prop:representable-code-polymatroid}, and \cref{lem:block-girth-rank} shows that $d^\perp(\mC)$ is the minimum size of a set of linearly dependent columns of $G$, that is, $d^\perp(\mC)=g(M)$. The condition $d^\perp(\mC)\geq3$ means that $M$ is simple, and \cref{thm:kung-block} becomes Kung's bound $\crit(M;q)\leq k-g(M)+3$.
\end{remark}
 
\begin{remark}\label{rem:kung-linear}
Let $\mC$ be a nondegenerate $\F_{q^h}$-linear code of $\F_{q^h}$-dimension $k'$, so that $k=hk'$. Every coordinate map of $\mC$ is $\F_{q^h}$-linear and nonzero, hence surjective, so $\mC$ is faithful.
The support of a subcode is the union of the supports of any generating set. Hence an $\F_{q^h}$-basis of an
$\F_{q^h}$-subcode with full support spans over $\F_q$ a subcode with full support, and conversely the
$\F_{q^h}$-span of an $\F_q$-subcode with full support has $\F_{q^h}$-dimension at most its
$\F_q$-dimension. Thus $\crit(\mC)$ is equal to the critical exponent of $\mC$ as an $\F_{q^h}$-linear code.
Similarly, the $\F_q$-rank of a set of blocks is $h$ times the $\F_{q^h}$-rank of the corresponding
columns, so by \cref{lem:block-girth-rank} $d^\perp(\mC)$ is the dual distance of $\mC$ as an
$\F_{q^h}$-linear code. If $3\leq d^\perp:=d^\perp(\mC)<\infty$, Kung's bound over $\F_{q^h}$ gives
$\crit(\mC)\leq k'-d^\perp+3$, whereas \cref{thm:kung-block} gives $\crit(\mC)\leq h(k'-d^\perp+1)+2$. The difference between the two bounds is $(h-1)(k'-d^\perp+1)$, which is nonnegative by \eqref{eq:dual-singleton}. So for $\F_{q^h}$-linear codes \cref{thm:kung-block} is implied by Kung's bound. For $h=1$ the two bounds coincide, and for $h>1$ they agree exactly when $d^\perp=k'+1$, that is, when $\mC$ is MDS.
\end{remark}
 
\begin{remark}\label{rem:kung-extremal}
Let $\mC$ be faithful with $h\geq2$ and $d^\perp(\mC)=3$. Then $\crit(\mC)=k-2h+2$ if and only if
$k=2h$ and $U_1,\ldots,U_n$ form a spread of $\F_q^{2h}$, as in \cref{ex:spread}. Indeed, by \cref{lem:block-girth-rank}, any two blocks span a subspace of dimension $2h$, so $k\geq2h$ and the $U_i$ pairwise intersect trivially. In particular $n\leq(q^k-1)/(q^h-1)$. Suppose $\crit(\mC)=k-2h+2$.
By \cref{prop:crit-codimension}, every subspace $W\subseteq\F_q^k$ of dimension $2h-1$ contains some $U_i$,
and it contains at most one, since two of them span a subspace of dimension $2h$. Write $\binom{a}{b}_q$ for the number of $b$-dimensional subspaces of $\F_q^a$. Double
counting the pairs $(U_i,W)$ with $U_i\subseteq W$ and $\dim W=2h-1$ gives
$$
\binom{k}{2h-1}_q=n\binom{k-h}{h-1}_q\leq\frac{q^k-1}{q^h-1}\binom{k-h}{h-1}_q .
$$
Since $\binom{k}{2h-1}_q\big/\binom{k-h}{h-1}_q=\prod_{j=k-h+1}^{k}(q^j-1)\big/\prod_{j=h}^{2h-1}(q^j-1)$,
this is equivalent to
$$
\prod_{j=k-h+1}^{k-1}(q^j-1)\leq\prod_{j=h+1}^{2h-1}(q^j-1).
$$
Both products have $h-1\geq1$ factors. If $k\geq2h+1$, each factor on the left is larger than the
corresponding factor on the right, a contradiction. Hence $k=2h$, equality holds throughout, and
$n=q^h+1$, so the $U_i$ form a spread. The converse is \cref{ex:spread}. For $h=1$, the first step of the argument shows that a simple matroid $M$ of rank $k\geq2$ represented over $\F_q$ by a set $S$ of points satisfies $\crit(M;q)=k$ if and only if every point of $\PG(k-1,q)$ lies in $S$, that is, $S=\PG(k-1,q)$. Thus Kung's bound with $g(M)=3$ is attained exactly by $\PG(k-1,q)$, for every $k\geq2$; for $k=1$, $\PG(0,q)$ is a single point and has no circuits.
\end{remark}

\section{Hypergraphic polymatroids and additive codes}\label{sec:hypergraphic}

In \cref{sec:polymatroid} the critical exponent, the minimum distance and the dual distance of an additive code were all expressed through its polymatroid $P_{\mC}$. In this section we look at a class of additive codes for which $P_{\mC}$ has a combinatorial description, namely the codes whose polymatroid is the hypergraphic polymatroid of \cite{whittle1992hypergraph}. For these codes the invariants of \cref{sec:polymatroid} become invariants of a hypergraph, so that results on hypergraphs give results on codes, and vice versa.

Throughout, $H=(V,\mE,\psi_H)$ is a hypergraph with $N$ vertices and $m$ edges, and we use the notation of \cref{subsec:hypergraphs}. In particular, for $A\subseteq\mE$, $\kappa(A)$ is the number of connected components of the spanning subhypergraph $H_A=(V,A,\psi_H|_A)$, isolated vertices included.

\subsection{The hypergraphic polymatroid and its representations}\label{subsec:canonical-representation}

We recall the hypergraphic polymatroid $f_H$, show that it is representable over every field by an explicit family of subspaces, and use it to define hypergraphic codes.

\begin{definition}\label{def:hypergraphic-function}
The \textbf{hypergraphic function} of $H$ is $f_H:2^{\mE}\to\mathbb Z_{\geq0}$, $f_H(A):=N-\kappa(A)$.
\end{definition}

If $C_1,\ldots,C_s$ are the vertex sets of the components of $H_A$, then $f_H(A)=\sum\limits_{j=1}^s(|C_j|-1)$. Isolated vertices contribute $0$. In particular $f_H(\emptyset)=0$ and $f_H(\{e\})=|e|-1$ for every edge $e$. For a graph $\mG$, $f_{\mG}$ is the rank function of the cycle matroid $M(\mG)$. Whittle showed that $f_H$ is the rank function of a polymatroid, namely the \emph{Dilworth truncation} of the function $A\mapsto|V(A)|$; see \cite[Theorem~5.1]{whittle1992hypergraph} and also \cite{lovasz1977flats}. This also follows from \cref{thm:canonical-hypergraphic-representation} below, since every function of the form $A\mapsto\dim\sum\limits_{e\in A}U_e$ is the rank function of a polymatroid.

Let $\mathbb K$ be a field, and let $\mathbb K^V$ be the $\mathbb K$-vector space of functions $V\to\mathbb K$, that is, of vectors $x=(x_v)_{v\in V}$ with coordinates indexed by the vertices. It is isomorphic to $\mathbb K^N$. Let $(\varepsilon_v)_{v\in V}$ be its standard basis, where $\varepsilon_v$ has entry $1$ at $v$ and $0$ elsewhere, and let
$$
Z:=\left\{x\in\mathbb K^V:\sum\limits_{v\in V}x_v=0\right\}.
$$
For every edge $e$ choose $v_e\in e$ and set
$$
U_e:=\langle\varepsilon_v-\varepsilon_{v_e}:v\in e\setminus\{v_e\}\rangle_{\mathbb K}=\langle\varepsilon_v:v\in e\rangle_{\mathbb K}\cap Z,
$$
so that $\dim U_e=|e|-1$.

\begin{theorem}\label{thm:canonical-hypergraphic-representation}
For every $A\subseteq \mE$ we have $\dim_{\mathbb K}\bigl(\sum\limits_{e\in A}U_e\bigr)=f_H(A)$. Consequently, $f_H$ is representable over every field $\mathbb K$. In particular, it is $\F_q$-representable for every prime power $q$.
\end{theorem}

\begin{proof}
Let $C_1,\ldots,C_s$ be the vertex sets of the components of $H_A$, so that $s=\kappa(A)$. For each~$j$ let $Z_{C_j}\subseteq\mathbb K^V$ be the space of vectors supported on $C_j$ with coordinate sum zero, which has dimension $|C_j|-1$. These spaces have pairwise disjoint supports, so their sum is direct. Every edge $e\in A$ is contained in a single component $C_j$, and then $U_e\subseteq Z_{C_j}$. Hence $\sum\limits_{e\in A}U_e\subseteq\bigoplus\limits_jZ_{C_j}$. Conversely, $U_e$ contains $\varepsilon_a-\varepsilon_b$ for all $a,b\in e$. Fix $j$ and $u_0\in C_j$. For every $u\in C_j$ there is a walk $u_0,e_1,u_1,\ldots,e_\ell,u_\ell=u$ with edges in $A$. Since $u_{i-1},u_i\in e_i$, we have $\varepsilon_{u_i}-\varepsilon_{u_{i-1}}\in U_{e_i}$ for every $i$, and therefore
$$
\varepsilon_u-\varepsilon_{u_0}=\sum\limits_{i=1}^\ell(\varepsilon_{u_i}-\varepsilon_{u_{i-1}})\in\sum\limits_{e\in A}U_e.
$$
The vectors $\varepsilon_u-\varepsilon_{u_0}$, $u\in C_j$, span $Z_{C_j}$. Hence $\sum\limits_{e\in A}U_e=\bigoplus\limits_jZ_{C_j}$, which has dimension $$\sum\limits_j(|C_j|-1)=N-\kappa(A)=f_H(A).$$
\end{proof}

The space $\langle\varepsilon_v:v\in e\rangle$ has dimension $|e|$, and $Z$ is a hyperplane containing none of the vectors~$\varepsilon_v$. Intersecting with $Z$ decreases the dimension of each of these spaces by one, and by \cref{thm:canonical-hypergraphic-representation} it gives a representation of $f_H$. This is exactly the hyperplane-section construction of Dilworth truncations in \cite{lovasz1977flats} and \cite[Sections~5-6]{whittle1992hypergraph}.

\Cref{thm:canonical-hypergraphic-representation} motivates the following definition.

\begin{definition}\label{def:hypergraphic-additive}
A nondegenerate additive code $\mC\subseteq\F_{q^h}^n$ is \textbf{hypergraphic} if there is a hypergraph $H=(V,[n],\psi_H)$ with $\rho_{\mC}=f_H$. Any such $H$ is called an \textbf{underlying hypergraph} of~$\mC$.
\end{definition}

If $H$ underlies $\mC$, then $\rho_{\mC}(\{i\})=f_H(\{i\})=|\psi_H(i)|-1$, and this is at least $1$ because $\mC$ is nondegenerate. Hence every edge of an underlying hypergraph of $\mC$ has at least two vertices.

\begin{example}\label{ex:small-code-hypergraphic}
The code $\mC$ of \cref{ex:small-code} is hypergraphic, with underlying hypergraph $H_0$ of \cref{ex:small-hypergraph}, where the edge $e_i$ corresponds to the coordinate $i\in[3]$. Indeed, a single edge $e_i$ has $3$ vertices, so $f_{H_0}(\{e_i\})=3-1=2$. Any two edges of $H_0$ have union $V=[4]$ and share the vertex $1$, so $f_{H_0}(A)=4-1=3$ for $|A|\geq2$. These values agree with $\rho_{\mC}$ computed in \cref{ex:small-code}.
\end{example}

In the language of polymatroids, Whittle showed that the hypergraphic polymatroids are exactly those obtained from the cycle matroid $M(K_N)$ of a complete graph by assigning to each element a nonempty connected flat \cite[Proposition~6.5]{whittle1992hypergraph}. Applied to $P_{\mC}$, this gives the following characterisation of hypergraphic codes.

\begin{theorem}[{\cite[Proposition~6.5]{whittle1992hypergraph}}]\label{thm:hypergraphic-characterisation}
A nondegenerate additive code $\mC\subseteq\F_{q^h}^n$ is hypergraphic if and only if there are an integer $N\geq2$ and nonempty connected flats $F_1,\ldots,F_n$ of the cycle matroid $M(K_N)$ such that $\rho_{\mC}(A)=r_{M(K_N)}\bigl(\bigcup_{i\in A}F_i\bigr)$ for every $A\subseteq[n]$. In this case $F_i=\binom{e_i}{2}$ for subsets $e_i\subseteq[N]$ with $|e_i|\geq2$, and $([N],[n],i\mapsto e_i)$ is an underlying hypergraph of $\mC$.
\end{theorem}

The last statement follows from the description of the flats of $M(K_N)$ as unions of the edge sets of complete graphs on the blocks of a partition of $[N]$; see \cite[Section~1.7]{oxley2011matroid}. Such a flat is nonempty and connected exactly when one block has at least two elements and all other blocks are singletons.

\begin{remark}\label{rem:not-hypergraphic}
Not every additive code is hypergraphic. For $h=1$, an underlying hypergraph is a graph, so a linear code is hypergraphic if and only if its matroid is graphic. For instance, the $[4,2,3]_3$ code, whose matroid is $U_{2,4}$, is not hypergraphic. For $h\geq2$, the classification in \cref{subsec:structural-nonexistence} shows that many faithful QMDS codes are not hypergraphic.
\end{remark}

\begin{corollary}\label{cor:uniformity-fullrank}
A hypergraphic code $\mC\subseteq\F_{q^h}^n$ is faithful if and only if some, or equivalently every, underlying hypergraph of $\mC$ is $(h+1)$-uniform.
\end{corollary}

\begin{proof}
For every underlying hypergraph $H$ we have $\rho_{\mC}(\{i\})=|\psi_H(i)|-1$ for every $i$, and $\mC$ is faithful if and only if $\rho_{\mC}(\{i\})=h$ for every $i$.
\end{proof}

\begin{remark}\label{rem:recover-hypergraph}
An underlying hypergraph need not be unique. For $h=1$, by Whitney's $2$-isomorphism theorem, two graphs have the same cycle matroid if and only if they are related by vertex identification, vertex cleaving and twisting \cite[Section~5.3]{oxley2011matroid}. For instance, all forests with~$m$ edges have the same cycle matroid. For this reason we state our results for every underlying hypergraph. Bounds involving parameters that depend on the choice of $H$, such as $\Delta(H)$ or $\alpha_{\mathrm w}(H)$, can then be applied to the underlying hypergraph that gives the best bound.
\end{remark}

\begin{remark}\label{rem:connected-wlog}
Every hypergraphic code has a connected underlying hypergraph. Indeed, let $H$ underlie $\mC$, and let $u$ and $w$ be vertices in different components of $H$. Let $H'$ be obtained from $H$ by identifying $u$ and $w$. For every $A\subseteq\mE$, the vertices $u$ and $w$ lie in different components of $H_A$, so $H'_A$ has exactly $\kappa(A)-1$ components, and $f_{H'}(A)=(N-1)-(\kappa(A)-1)=f_H(A)$. Moreover, $H'$ has the same edge sizes as $H$, and it has the same Berge cycles, since identifying two nodes in different components of the incidence graph creates no cycle. Repeating this, we obtain a connected underlying hypergraph of $\mC$. Hence  we can assume that the underlying hypergraph is connected.
\end{remark}

\subsection{The additive code of a uniform hypergraph}\label{subsec:canonical-code}

We show that every uniform hypergraph is an underlying hypergraph of a faithful additive code, obtained by evaluating differences of vertex labels on the edges.

Let $H=(V,\mE,\psi_H)$ be $r$-uniform with $r\geq2$, and set $h:=r-1$. For every edge fix an ordering $e=(v_{e,1},\ldots,v_{e,r})$ and let $B_e\in\Mat_{N,h}(\F_q)$ be the matrix with columns $\varepsilon_{v_{e,j}}-\varepsilon_{v_{e,1}}$, $2\leq j\leq r$, so that $\colsp_{\F_q}(B_e)=U_e$. Fix an $\F_q$-linear bijection $\iota:\F_q^h\to\F_{q^h}$.

\begin{definition}\label{def:hypergraphic-code}
For a vertex labelling $a:V\to\F_q$, viewed as a row vector in $\F_q^V$, let
$$\Phi_H(a):=\bigl(\iota(aB_e)\bigr)_{e\in \mE}\in\F_{q^h}^m,$$
where
$$aB_e=\bigl(a(v_{e,j})-a(v_{e,1})\bigr)_{2\leq j\leq r}\in\F_q^h.
$$
The \textbf{hypergraphic code of $H$} is $\mC_H:=\im(\Phi_H)$.
\end{definition}

Another ordering of an edge, or another bijection $\iota$, changes $\mC_H$ by an $\F_q$-linear bijection of the alphabet in the corresponding coordinate, so all choices give equivalent codes.

\begin{theorem}\label{thm:hypergraphic-rank}
For every $A\subseteq \mE$ we have $\rho_{\mC_H}(A)=f_H(A)$. In particular, $\mC_H$ is a faithful hypergraphic code with underlying hypergraph $H$, and $\dim_{\F_q}\mC_H=N-\kappa(H)$.
\end{theorem}

\begin{proof}
Let $B:=(B_e)_{e\in \mE}\in\Mat_{N,mh}(\F_q)$ and $B_A:=(B_e)_{e\in A}$. Identifying each coordinate with $\F_q^h$ through $\iota$, we have $\mC_H=\{aB:a\in\F_q^V\}$. The kernel of the projection of $\mC_H$ onto the coordinates in $A$ is $\mC_H(A)$, so $\rho_{\mC_H}(A)=\rank(B_A)=\dim\sum\limits_{e\in A}U_e$, which is equal to  $f_H(A)$ by \cref{thm:canonical-hypergraphic-representation}. Taking $A=\{e\}$ gives $\rho_{\mC_H}(\{e\})=h$, and taking $A=\mE$ gives the dimension.
\end{proof}

The uniformity assumption is needed only for the alphabet. An additive code has the same alphabet $\F_{q^h}$ in every coordinate, whereas the space $U_e$ attached to an edge $e$ has dimension $|e|-1$. Hence all edges must have the same size for all coordinates of $\mC_H$ to be identified with $\F_{q^h}$ through $\iota$. If $H$ is not uniform but every edge has at least two vertices, one can take $h:=\max\limits_{e\in\mE}(|e|-1)$ and add $h-|e|+1$ zero columns to each $B_e$. The resulting code still satisfies $\rho_{\mC_H}=f_H$, so it is hypergraphic with underlying hypergraph $H$, but it is not faithful. \Cref{thm:chromatic-critical} below applies to it, whereas the results that assume faithfulness do not.

\begin{example}\label{ex:fano-code}
For the Fano plane $H$ of \cref{ex:fano}, $\mC_H$ is a faithful additive code in $\F_{q^2}^7$ of $\F_q$-dimension $6$, that is, a $[7,3]_q^2$ code, for every prime power $q$.
\end{example}

\subsection{Weak colourings and the critical exponent}\label{subsec:weak-coloring}

We show that the critical exponent of a hypergraphic code is determined by the weak chromatic number of an underlying hypergraph.

The following identity goes back to Helgason \cite{helgason1974} and appears as \cite[Theorem~4.1]{whittle1992hypergraph}, and the critical-exponent formulation is \cite[Theorem~7.1]{whittle1992hypergraph}. We include the proof for completeness.

\begin{theorem}[Helgason-Whittle]\label{thm:chromatic-critical}
Let $\mC$ be a hypergraphic additive code with underlying hypergraph $H$. Then $\chi_H^{\mathrm w}(z)=z^{\kappa(H)}p_{\mC}(z)$ and $\crit(\mC)=\lceil\log_q\chi_{\mathrm w}(H)\rceil$.
\end{theorem}

\begin{proof}
For $A\subseteq \mE$, the maps $V\to[z]$ that are constant on every edge of $A$ are those that are constant on the components of $H_A$. There are $z^{\kappa(A)}$ of them. Inclusion-exclusion gives $\chi_H^{\mathrm w}(z)=\sum\limits_{A\subseteq \mE}(-1)^{|A|}z^{\kappa(A)}$, which also shows that $\chi_H^{\mathrm w}$ is a polynomial. Since $\rho_{\mC}=f_H$, we have $k-\rho_{\mC}(A)=f_H(\mE)-f_H(A)=\kappa(A)-\kappa(H)$, whence $p_{\mC}(z)=z^{-\kappa(H)}\chi_H^{\mathrm w}(z)$. By \cref{cor:critical-full-support}, $\crit(\mC)$ is the least $t\geq1$ such that $H$ has a weak colouring with $q^t$ colours. Every edge has at least two vertices, so $\chi_{\mathrm w}(H)$ is finite, and a weak colouring with $s$ colours is also one with $s'\geq s$ colours, so this least $t$ is $\lceil\log_q\chi_{\mathrm w}(H)\rceil$.
\end{proof}

\begin{example}\label{ex:fano-crit}
For the Fano plane, \cref{ex:fano} and \cref{thm:chromatic-critical} give $p_{\mC_H}(z)=(z-1)(z-2)(z^4+3z^3-6z+3)$. Hence $\crit(\mC_H)=2$ for $q=2$ and $\crit(\mC_H)=1$ for $q\geq3$. For the code of \cref{ex:small-code}, $\chi^{\mathrm w}_{H_0}(z)=z(z-1)^2(z+2)$.
\end{example}

For the hypergraphic code of a uniform hypergraph, the identity can be seen directly from vertex labellings.

\begin{proposition}\label{prop:colorings-full-support}
Let $H$ be $r$-uniform, fix $t\geq1$, and identify a set of $q^t$ colours with $\F_q^t$. Consider the map $(a_1,\ldots,a_t)\mapsto(\Phi_H(a_1),\ldots,\Phi_H(a_t))$ from $t$-tuples of vertex labellings $a_j:V\to\F_q$ to $\mC_H^t$.
\begin{enumerate}[label=\textup{(\roman*)}]
\item The labelling $\gamma(v):=(a_1(v),\ldots,a_t(v))$ is a weak colouring of $H$ if and only if the image $t$-tuple has full joint support.
\item Every $t$-tuple of codewords has exactly $q^{t\kappa(H)}$ preimages.
\end{enumerate}
Consequently, $\chi_H^{\mathrm w}(q^t)=q^{t\kappa(H)}p_{\mC_H}(q^t)$.
\end{proposition}

\begin{proof}
The $e$-th coordinate of $\Phi_H(a_j)$ vanishes if and only if $a_j$ is constant on $e$. Hence all $t$ codewords vanish at $e$ if and only if $\gamma$ is constant on $e$, which proves (i). The kernel of $\Phi_H$ consists of the labellings that are constant on the components of $H$, so it has dimension $\kappa(H)$, which proves (ii). The final identity follows from \cref{cor:critical-full-support}.
\end{proof}

\subsection{Berge girth and the Kung-type bound}\label{subsec:berge-kung}

We show that the dual distance of a faithful hypergraphic code is the Berge girth of every underlying hypergraph, and we use this to specialise the Kung-type bound of \cref{thm:kung-block}.

\begin{theorem}\label{thm:dual-berge-girth}
Let $\mC$ be a faithful hypergraphic additive code and let $H$ be any underlying hypergraph. Then $d^\perp(\mC)=g_{\rm B}(H)$.
\end{theorem}

\begin{proof}
By \cref{cor:uniformity-fullrank}, $H$ is $(h+1)$-uniform. For $A\subseteq \mE$ let $I_A$ be the incidence graph of $H_A$, with node set $V\sqcup A$. Its cycles of length $2\ell$ are exactly the Berge cycles of $H$ with $\ell$ edges, all in $A$, and its components correspond to the components of $H_A$. Consider a component of $I_A$ with $a$ nodes in $A$ and $v$ nodes in $V$. Every node in $A$ has degree $h+1$, so the component has $(h+1)a$ edges and $a+v$ nodes. If it is a tree, then $(h+1)a=a+v-1$, and hence $v-1=ha$. Otherwise it contains a cycle, so $(h+1)a\geq a+v$, and hence $v-1\leq ha-1$. Isolated vertices of $H_A$ are components with $a=0$ and $v=1$, for which $v-1=ha=0$. Summing $v-1$ and $ha$ over all components gives
$$
f_H(A)=N-\kappa(A)\leq h|A|,
$$
with equality if and only if no Berge cycle has all its edges in $A$. Since $\rho_{\mC}=f_H$, \cref{lem:block-girth-rank} gives $d^\perp(\mC)=\min\{|A|:A\text{ contains the edges of a Berge cycle}\}=g_{\rm B}(H)$.
\end{proof}

For a faithful hypergraphic code with a connected underlying hypergraph that has a Berge cycle, the dual Singleton bound \eqref{eq:dual-singleton} and \cref{thm:dual-berge-girth} give $g_{\rm B}(H)\leq\lfloor(N-1)/h\rfloor+1$. This can also be seen directly: a set $A$ of $g_{\rm B}(H)-1$ edges contains no Berge cycle, so $h|A|=f_H(A)\leq N-1$ by the proof of \cref{thm:dual-berge-girth}. For the Fano plane the bound reads $3\leq4$.

\begin{corollary}\label{cor:kung-hypergraph}
Let $\mC$ be a faithful hypergraphic additive code with a connected underlying hypergraph $H$ on $N$ vertices. If $H$ is linear and has a Berge cycle, that is, $3\leq g_{\rm B}(H)<\infty$, then
$$
\crit(\mC)=\lceil\log_q\chi_{\mathrm w}(H)\rceil\leq N-h\bigl(g_{\rm B}(H)-1\bigr)+1,$$
equivalently,
$$\chi_{\mathrm w}(H)\leq q^{N-h(g_{\rm B}(H)-1)+1}.
$$
\end{corollary}

\begin{proof}
Since $H$ is connected, $\dim_{\F_q}\mC=f_H(\mE)=N-1$. By \cref{thm:dual-berge-girth}, $d^\perp(\mC)=g_{\rm B}(H)$, and \cref{thm:kung-block} with $k=N-1$ gives the bound. The equality is \cref{thm:chromatic-critical}.
\end{proof}

For the Fano plane, $g_{\rm B}(H)=3$ and \cref{cor:kung-hypergraph} gives $\crit(\mC_H)\leq4$, while $\crit(\mC_H)\leq2$ by \cref{ex:fano-crit}. For hypergraphic codes, \cref{cor:kung-hypergraph} is usually much weaker than the bounds of \cref{cor:crit-hypergraph-bounds} below; see \cref{tab:crit-hypergraphic}. For the code of \cref{ex:small-code}, $d^\perp(\mC)=g_{\rm B}(H_0)=2$, in accordance with \cref{thm:dual-berge-girth}.

\begin{corollary}\label{cor:crit-hypergraph-bounds}
Let $\mC$ be a hypergraphic additive code with underlying hypergraph $H$. Then
$$
\Bigl\lceil\log_q\Bigl\lceil\frac{N}{\alpha_{\mathrm w}(H)}\Bigr\rceil\Bigr\rceil\leq\crit(\mC)\leq\lceil\log_q(\Delta(H)+1)\rceil.
$$
If moreover $\mC$ is faithful, then $\crit(\mC)\leq\min\{\dim_{\F_q}\mC-h+1,\lfloor(\log_q m)/h\rfloor+1\}$.
\end{corollary}

\begin{proof}
Every colour class of a weak colouring contains no edge, so $N\leq\chi_{\mathrm w}(H)\,\alpha_{\mathrm w}(H)$. For the upper bound, colour the vertices one by one in any order. When $v$ is coloured, an edge containing $v$ excludes at most one colour, namely when all its other vertices already have that same colour. Hence at most $\Delta(H)$ colours are excluded and $\chi_{\mathrm w}(H)\leq\Delta(H)+1$. Both bounds now follow from \cref{thm:chromatic-critical}. The last statement is \cref{prop:crit-general-bounds} with $\rho_{\min}=h$ and $n=m$.
\end{proof}

\section{Minimum distance and hypergraph connectivity}\label{sec:distance}

In this section we show that the minimum distance of a hypergraphic code with a connected underlying hypergraph is equal to the edge-connectivity of the hypergraph. We then combine this identity with the Moore bound, and we classify the hypergraphic QMDS codes.

We recall the notation of \cref{subsec:hypergraphs}. For $\emptyset\neq S\subsetneq V$, the edge cut $\delta_H(S)$ is the set of edges meeting both $S$ and $V\setminus S$, and for a connected hypergraph $H$ the edge-connectivity $\lambda(H)$ is the minimum size of an edge cut. We write $\delta_{\min}(H)$ for the minimum degree. Taking $S=\{v\}$ for a vertex $v$ of minimum degree gives
\begin{equation}\label{eq:lambda-delta}
\lambda(H)\leq\delta_{\min}(H),
\end{equation}
which we use repeatedly. Finally, if $\mC$ is hypergraphic with a connected underlying hypergraph~$H$, then $\dim_{\F_q}\mC=f_H(\mE)=N-1$.

\begin{theorem}\label{thm:distance-connectivity}
Let $\mC$ be a hypergraphic additive code with a connected underlying hypergraph $H$. Then $d(\mC)=\lambda(H)$. In particular, if $H$ is connected and $r$-uniform, then $\mC_H$ is an $\bigl[m,\frac{N-1}{r-1},\lambda(H)\bigr]_q^{r-1}$ code.
\end{theorem}

\begin{proof}
Since $H$ is connected, $k=f_H(\mE)=N-1$, and $\rho_{\mC}(A)<k$ if and only if $f_H(A)=N-\kappa(A)<N-1$, that is, if and only if $H_A$ is disconnected. By \cref{lem:distance-rank}, $d(\mC)$ is the minimum number of edges whose deletion disconnects $H$. For every $S$, deleting $\delta_H(S)$ separates $S$ from $V\setminus S$. Conversely, if deleting $R\subseteq \mE$ disconnects $H$ and $S$ is the vertex set of a component of $H_{\mE\setminus R}$, then $\delta_H(S)\subseteq R$. Hence the minimum is equal to  $\lambda(H)$. The parameters of $\mC_H$ follow from \cref{thm:hypergraphic-rank}.
\end{proof}

\begin{example}\label{ex:distance-examples}
For the Fano plane, $\lambda(H)=3$ by \cref{ex:fano}, so $\mC_H$ is a $[7,3,3]_q^2$ code for every~$q$. For the code of \cref{ex:small-code}, $d(\mC)=\lambda(H_0)=2$.
\end{example}

\begin{remark}[Graphs]\label{rem:graphic-case}
Let $\mG=(V,\mE)$ be a connected graph, with an orientation
chosen for each edge. In the construction of $\mC_{\mG}$,
an edge $e$ has two vertices, so $h=1$, and its coordinate
is the difference $a(v)-a(u)$ for $e$ oriented from $u$
to $v$. The code $\mC_{\mG}$ is the cut space of $\mG$:
it is generated by the codewords obtained from vertex
labellings that are $1$ at one vertex and $0$ elsewhere.
Its dual is the cycle space, since the relations among
the columns of the vertex-edge incidence matrix are
the flows with zero sum at every vertex. Thus
\cref{thm:distance-connectivity,thm:dual-berge-girth}
give $d(\mC_{\mG})=\lambda(\mG)$ and
$d^\perp(\mC_{\mG})=g(\mG)$, the classical cut-space and
cycle-space identities \cite{hakimi1968graph}. Moreover,
a weak colouring of a graph is a proper colouring, so
\cref{thm:chromatic-critical} gives
$\crit(\mC_{\mG})=\lceil\log_q\chi(\mG)\rceil$.
\end{remark}

\subsection{A girth-distance bound}

The Moore bound for hypergraphs
(\cite{hoory2002size,ellis2014regular,erskine2022small})
gives a relation between Berge girth, minimum degree and
the number of vertices. Together with the distance
identities above, it gives the following code bound. We include the short counting argument.

\begin{proposition}[Girth-distance bound]\label{prop:moore}
Let $\mC$ be a faithful hypergraphic additive code of
$\F_q$-dimension $k$ with connected underlying
$(h+1)$-uniform hypergraph $H$, let $d:=d(\mC)$, and let $t\geq1$. With the convention $0^0:=1$, the following hold.
\begin{enumerate}[label=\textup{(\roman*)}]
\item If $d^\perp(\mC)\geq2t+1$, then $\displaystyle k\geq h\,d\sum\limits_{j=0}^{t-1}\bigl(h(d-1)\bigr)^j$.
\item If $d^\perp(\mC)\geq2t$, then $\displaystyle k\geq(h+1)\sum\limits_{j=0}^{t-1}\bigl(h(d-1)\bigr)^j-1$.
\end{enumerate}
In particular, $k\geq hd$ if $d^\perp(\mC)\geq3$, $k\geq h+h(h+1)(d-1)$ if $d^\perp(\mC)\geq4$, and $k\geq hd+h^2d(d-1)$ if $d^\perp(\mC)\geq5$.
\end{proposition}

\begin{proof}
Write $\delta:=\delta_{\min}(H)$ and $g:=g_{\rm B}(H)$, so that $g=d^\perp(\mC)$ by \cref{thm:dual-berge-girth}. The incidence graph $I$ of $H$ is bipartite and has girth $2g$. Fix a node $x$ of $I$ and an integer $\rho$ with $2\rho<2g$. If a node at distance $s+1\leq\rho$ from $x$ had two neighbours at distance $s$, then $I$ would contain a cycle of length at most $2\rho$. Hence every node at distance $s<\rho$ from $x$ has at least its degree minus one neighbours at distance $s+1$, and distinct nodes at distance $s$ have disjoint sets of such neighbours. Every vertex of $H$ has degree at least $\delta$ in $I$, and every edge of $H$ has degree $h+1$.

For (i), take for $x$ a vertex of $H$ and $\rho=2t\leq g-1$. There are at least $\delta$ nodes at distance $1$, and at least $h\delta\bigl(h(\delta-1)\bigr)^{j-1}$ vertices of $H$ at distance $2j$ for $1\leq j\leq t$. Hence $$N\geq1+h\delta\sum\limits_{j=0}^{t-1}(h(\delta-1))^j.$$ 
For (ii), take for $x$ an edge of $H$ and $\rho=2t-1\leq g-1$. There are $h+1$ vertices at distance $1$, and at least $(h+1)\bigl(h(\delta-1)\bigr)^j$ vertices at distance $2j+1$ for $0\leq j\leq t-1$. Hence $$N\geq(h+1)\sum\limits_{j=0}^{t-1}(h(\delta-1))^j.$$ 
In both cases the right-hand side is nondecreasing in $\delta\geq1$. Since $k=N-1$ and $d=\lambda(H)\leq\delta$ by \cref{thm:distance-connectivity} and \eqref{eq:lambda-delta}, the claims follow.
\end{proof}

For given $k$, $h$ and $d^\perp(\mC)$, \cref{prop:moore} bounds the minimum distance of every faithful hypergraphic code, and it is attained by several classical geometries. If $H$ is a Steiner system $S(2,h+1,N)$, that is, if every pair of vertices lies in exactly one edge, then $H$ is linear, has a Berge cycle of length $3$ when $N>h+1$, and is regular of degree $(N-1)/h$. Hence (i) with $t=1$ reads $d(\mC_H)\leq(N-1)/h=\delta_{\min}(H)$, and equality holds whenever $\lambda(H)=\delta_{\min}(H)$. This is the case for every projective plane of order $h\geq2$: if $V=S\sqcup T$ with $S,T\neq\emptyset$, then $S$ and $T$ cannot both contain a line, since two lines always meet; if $T$ contains no line, then the $h+1$ lines through a vertex of $T$ all lie in $\delta_H(S)$. Hence the code of a projective plane of order $h$ is an $[h^2+h+1,h+1,h+1]_q^h$ code with $d^\perp=3$ that attains (i). For Steiner triple systems see \cref{cor:sts}. Similarly, a generalised quadrangle with $h+1$ points on each line and $\delta$ lines through each point has $N=(h+1)(h(\delta-1)+1)$ points and Berge girth $4$, and a generalised hexagon has $N=(h+1)\bigl(1+h(\delta-1)+h^2(\delta-1)^2\bigr)$ points and Berge girth $6$. So (ii), with $t=2$ and $t=3$ respectively, is attained whenever $\lambda(H)=\delta$; by \cref{tab:distance-hypergraphic} this holds, for instance, for the generalised quadrangles of orders $(2,1)$, $(2,2)$, $(2,4)$ and $(4,2)$ and for the split Cayley hexagon of order $(2,2)$.

The girth-distance bound can also exclude parameters that the degree bound allows. If $h=2$, $N=20$ and $g_{\rm B}(H)\geq5$,
then $d(\mC)\leq2$. Indeed, if $d(\mC)\geq3$, (i) with $t=2$ would
give $k\geq2\cdot3+2^2\cdot3\cdot2=30$, whereas
$k=N-1=19$. The degree bound
$d(\mC)\leq\lfloor3m/20\rfloor$ does not exclude
$d(\mC)=3$ when $m\geq20$.

\subsection{Hypergraphic QMDS codes}\label{subsec:structural-nonexistence}

We now determine the hypergraphic QMDS codes with a connected underlying hypergraph. We first refine the Singleton bound \eqref{eq:singleton} in terms of the minimum number of coordinates on which the projection of the code is injective. For an additive code $\mC$ of $\F_q$-dimension $k$ let $\tau(\mC):=\min\{|A|:A\subseteq[n],\ \rho_{\mC}(A)=k\}$. For a hypergraphic code with a connected underlying hypergraph, $\tau(\mC)$ is the minimum number of edges of a connected spanning subhypergraph.

\begin{lemma}\label{lem:information-set}
Every nonzero additive code $\mC\subseteq\F_{q^h}^n$ of $\F_q$-dimension $k$ satisfies $d(\mC)\leq n-\tau(\mC)+1$ and $\tau(\mC)\geq\lceil k/h\rceil$. In particular, it satisfies the Singleton bound \eqref{eq:singleton}.
\end{lemma}

\begin{proof}
Let $A$ attain $\tau(\mC)$ and $a\in A$. By minimality, $\rho_{\mC}(A\setminus\{a\})<k$, so \cref{lem:distance-rank} gives $n-d(\mC)\geq\tau(\mC)-1$. Moreover $k=\rho_{\mC}(A)\leq\sum\limits_{i\in A}\rho_{\mC}(\{i\})\leq h\tau(\mC)$.
\end{proof}

For a connected hypergraph $H$, let $\tau(H)$ be the minimum number of edges of a connected spanning subhypergraph $H_A$.

\begin{proposition}\label{prop:spanning-singleton}
Let $\mC$ be a faithful hypergraphic additive code with a connected underlying hypergraph $H$ with $N$ vertices and $m$ edges. Then
$$
d(\mC)=\lambda(H)\leq m-\tau(H)+1\leq m-\Bigl\lceil\frac{N-1}{h}\Bigr\rceil+1,
$$
and the rightmost term is the Singleton bound for $\mC$.
\end{proposition}

\begin{proof}
Since $\rho_{\mC}(A)=N-\kappa(A)$ is equal to  $k=N-1$ exactly when $H_A$ is connected, $\tau(\mC)=\tau(H)$. The equality is \cref{thm:distance-connectivity}, and both inequalities follow from \cref{lem:information-set} with $n=m$ and $k=N-1$.
\end{proof}

For the rest of this subsection, $\mC$ is a faithful hypergraphic additive code and $H=(V,\mE,\psi_H)$ is a connected underlying hypergraph of $\mC$ with $N$ vertices and $m$ edges. By \cref{cor:uniformity-fullrank}, $H$ is $(h+1)$-uniform. Set $\sigma:=\lceil(N-1)/h\rceil$. Since $\dim_{\F_q}\mC=N-1$, the Singleton bound reads $d(\mC)\leq m-\sigma+1$, and $N-1=\rho_{\mC}(\mE)\leq hm$ gives $m\geq\sigma$.

\begin{proposition}\label{prop:singleton-structure}
If $\mC$ is QMDS, then $\tau(H)=\sigma$, $\lambda(H)=m-\sigma+1$, and
\begin{equation}\label{eq:singleton-degree-count}
N(m-\sigma+1)\leq N\delta_{\min}(H)\leq(h+1)m.
\end{equation}
\end{proposition}

\begin{proof}
If $\mC$ is QMDS, equality holds throughout \cref{prop:spanning-singleton}, which gives $\tau(H)=\sigma$ and $\lambda(H)=m-\sigma+1$. The first inequality in \eqref{eq:singleton-degree-count} is \eqref{eq:lambda-delta}, and the second holds because the degrees of the $(h+1)$-uniform hypergraph $H$ sum to $(h+1)m$.
\end{proof}

\begin{theorem}[Hypergraphic QMDS codes]\label{thm:hypergraphic-singleton-classification}
Assume $h\geq2$. With the notation above, the following hold.
\begin{enumerate}[label=\textup{(\roman*)}]
\item If $\sigma=1$, then $N=h+1$, every edge is equal to  $V$, and $\mC$ is QMDS with $d(\mC)=m$.
\item If $\sigma=2$, set $\overline{e}:=V\setminus e$ for $e\in \mE$. Then $\mC$ is QMDS if and only if the sets $\overline{e}$, $e\in\mE$, are pairwise disjoint, and in that case $d(\mC)=m-1$.
\item If $\sigma\geq3$, then $\mC$ is QMDS if and only if either $m=\sigma$, in which case $d(\mC)=1$, or $\sigma=3$, $m=4$, $N=2h+2$, every vertex lies in exactly two edges, and deleting any one edge leaves a connected spanning hypergraph, in which case $d(\mC)=2$.
\end{enumerate}
Consequently, if $N>h+1$, every QMDS code in this class with $d(\mC)\geq2$ has length $m\leq h+2$. This bound is attained for every $h\geq2$ and every prime power $q$.
\end{theorem}

\begin{proof}
If $\sigma=1$, then $N\leq h+1$, and since edges have $h+1$ vertices, $N=h+1$ and every edge is equal to  $V$. Deleting fewer than $m$ edges leaves a connected spanning hypergraph, so $\lambda(H)=m=m-\sigma+1$.

Suppose that $\sigma=2$, so $h+2\leq N\leq2h+1$. A single edge does not span $V$, so $\lambda(H)\leq m-1$, and $\mC$ is QMDS if and only if every pair of edges forms a connected spanning hypergraph. Any two $(h+1)$-subsets of $V$ intersect because $2(h+1)>N$, so a pair of edges forms a connected spanning hypergraph if and only if its union is $V$, that is, if and only if $\overline{e}\cap\overline{f}=\emptyset$. This proves (ii).

Let $\sigma\geq3$. If $m=\sigma$, then $\tau(H)\geq\sigma=m$ and \cref{prop:spanning-singleton} give $\lambda(H)\leq1$. Since $H$ is connected, $\lambda(H)=1=m-\sigma+1$, so $\mC$ is QMDS. It remains to treat QMDS codes with $m\geq\sigma+1$, that is, with $d(\mC)\geq2$. The definition of $\sigma$ gives $N\geq(\sigma-1)h+2>h+1$, and \eqref{eq:singleton-degree-count} rearranges to $m(N-h-1)\leq N(\sigma-1)$. Using $m\geq\sigma+1$ we obtain $2N\leq(h+1)(\sigma+1)$, hence $2\bigl((\sigma-1)h+2\bigr)\leq(h+1)(\sigma+1)$, which is equivalent to $(\sigma-3)(h-1)\leq0$. As $h\geq2$, this forces $\sigma=3$. Then $N\geq2h+2$, and \eqref{eq:singleton-degree-count} gives $(2h+2)(m-2)\leq N(m-2)\leq(h+1)m$, hence $m\leq4$. Thus $m=4$, and equality forces $N=2h+2$. Since $\lambda(H)=2$, every vertex has degree at least two by \eqref{eq:lambda-delta}, and the degrees sum to $4(h+1)=2N$, so every vertex has degree exactly two. Moreover, deleting one edge leaves $H$ connected. Conversely, under these conditions $\lambda(H)\geq2$, and isolating a vertex shows $\lambda(H)\leq2$, so $d(\mC)=2=m-\sigma+1$. This proves (iii).

For the length bound, let $N>h+1$ and $d(\mC)\geq2$, so $\sigma\geq2$ and $m\neq\sigma$ when $\sigma\geq3$. If $\sigma\geq3$, then $m=4\leq h+2$ by (iii). If $\sigma=2$, write $N=h+1+s$ with $1\leq s\leq h$. The pairwise disjoint sets $\overline{e}$ have size $s$, so $ms\leq h+1+s$ and $m\leq1+(h+1)/s\leq h+2$. Let $|V|=h+2$ and take all $(h+1)$-subsets of $V$ as edges. Their complements are distinct singletons, so by (ii) the code $\mC_H$ is QMDS with $m=h+2$ and $d=h+1$.
\end{proof}

The hypothesis $N>h+1$ in the length bound cannot be dropped. In fact, for $\sigma=1$, taking $m$ copies of the edge $V$ gives QMDS codes with $d=m$ for every $m$.

In the exceptional case of \cref{thm:hypergraphic-singleton-classification}(iii) with $m=4$, every vertex of $H$ lies in exactly two of the four edges. Such a hypergraph is determined by recording, for each vertex, the pair of edges containing it, that is, by a graph on the four edges. We make this precise and show that the case occurs for every $h\geq2$ and over every finite field. For a graph $\mG$ and a vertex $i$, we write $\mG-i$ for the graph obtained by deleting $i$ and the edges incident with it.

\begin{proposition}[The four-edge family]\label{prop:four-edge-family}
Let $h\geq2$ and $r:=h+1$. For a $2$-uniform hypergraph $\mG$ on $[4]$, that is, a graph with possibly parallel edges, let $H_{\mG}$ be its \textbf{dual hypergraph}: its vertices are the edges of $\mG$, and its edges are $e_i:=\{\text{edges of $\mG$ incident with }i\}$ for $i\in[4]$.
\begin{enumerate}[label=\textup{(\alph*)}]
\item The $r$-uniform hypergraphs with four edges in which every vertex lies in exactly two edges and deleting any one edge leaves a connected spanning hypergraph are, up to isomorphism, exactly the $H_{\mG}$ with $\mG$ $r$-regular and $\mG-i$ connected for every $i\in[4]$.
\item Such a $\mG$ exists for every $h\geq2$. Hence for every $h\geq2$ and every prime power $q$ there is a QMDS $[4,(2h+1)/h,2]_q^h$ code $\mC_{H_{\mG}}$.
\end{enumerate}
\end{proposition}

\begin{proof}
Let $H$ have edges $e_1,\ldots,e_4$ and suppose that every vertex lies in exactly two of them. Let $\mG$ be the graph on $[4]$ with one edge $\{i,j\}$ for each vertex of $H$ lying in $e_i\cap e_j$. Then $H\cong H_{\mG}$, and conversely every vertex of $H_{\mG}$ lies in exactly the two edges indexed by its endpoints. Since $|e_i|=\deg_{\mG}(i)$, the hypergraph $H_{\mG}$ is $r$-uniform if and only if $\mG$ is $r$-regular, and then $N=|E(\mG)|=2r=2h+2$. Every edge of $\mG$ has an endpoint different from $i$, so the edges $e_j$, $j\neq i$, cover all vertices of $H_{\mG}$. Moreover $e_j\cap e_l$ is the set of edges of $\mG$ joining $j$ and $l$, so deleting $e_i$ from $H_{\mG}$ leaves a connected spanning hypergraph if and only if $\mG-i$ is connected. This proves (a).

For (b), if $r$ is even take $r/2$ copies of a $4$-cycle, and if $r$ is odd take $K_4$ together with $(r-3)/2$ copies of a $4$-cycle. In both cases $\mG$ is $r$-regular and $\mG-i$ is connected for every $i$. The hypergraph $H_{\mG}$ is connected, has $N=2h+2$ vertices, so $\sigma=\lceil(2h+1)/h\rceil=3$, and by (a) and \cref{thm:hypergraphic-singleton-classification}(iii) the code $\mC_{H_{\mG}}$ is QMDS with $d=2$. Its $\F_q$-dimension is $N-1=2h+1$.
\end{proof}

\begin{example}\label{ex:qmds-examples}
For $h=2$ and $\mG=K_4$, the dual hypergraph $H_{\mG}$ is the Pasch configuration: its six points are the edges of $K_4$, its four lines are the stars of the vertices of $K_4$, and every point lies on exactly two lines. By \cref{prop:four-edge-family}, $\mC_{H_{\mG}}$ is a QMDS $[4,5/2,2]_q^2$ code for every $q$. For $h=3$, partition a six-element set $V$ into three pairs $P_1,P_2,P_3$ and let $H_1$ be the hypergraph on $V$ with edges $e_i:=V\setminus P_i$. By \cref{thm:hypergraphic-singleton-classification}(ii), $\mC_{H_1}$ is a QMDS $[3,5/3,2]_q^3$ code. At the maximal nontrivial length, $H=K_{h+2}^{(h+1)}$ gives a QMDS $[h+2,(h+1)/h,h+1]_q^h$ code. For $h=2$ this is $K_4^{(3)}$, which gives a QMDS $[4,3/2,3]_q^2$ code for every $q$. The code of \cref{ex:small-code} is QMDS by (ii), since the complements of the edges of $H_0$ are $\{4\},\{3\},\{2\}$. The Fano code of \cref{ex:distance-examples} has $\sigma=3$ and $m=7$, so it is not QMDS.
\end{example}

\section{Further bounds}\label{sec:lp}

In this section we look at two bounds that hold for every additive code, the Griesmer bound and the linear programming bound. For hypergraphic codes they become necessary conditions on the parameters $N$, $m$, $\lambda(H)$ and $g_{\rm B}(H)$ of a hypergraph. We also bound the minimum distance from below by an eigenvalue of a matrix attached to the hypergraph. Finally, we compare all these bounds on examples, both for hypergraphic codes and for known optimal additive codes.

\subsection{Griesmer and linear programming bounds}\label{subsec:delsarte}

For positive integers $k$ and $D$ let $g_q(k,D):=\sum\limits_{i=0}^{k-1}\lceil D/q^i\rceil$. The Griesmer bound states that every linear $[n,k,D]_q$ code satisfies $n\geq g_q(k,D)$ \cite{huffman2003fundamentals}. Its additive version is the following.

\begin{theorem}[Additive Griesmer bound {\cite[Theorem~12]{ball2025griesmer}}]\label{thm:additive-griesmer}
Every faithful additive $[n,k/h,d]_q^h$ code satisfies $[h]_qn\geq g_q(k,q^{h-1}d)$.
\end{theorem}

By \cref{thm:distance-connectivity}, if $\mC$ is a faithful hypergraphic additive code with a connected underlying hypergraph $H$ on $N$ vertices and $m$ edges, this reads
$$
g_q\bigl(N-1,q^{h-1}\lambda(H)\bigr)\leq[h]_qm.
$$
Stronger bounds can be obtained by replacing each subspace of $\mX(\mC)$ by the $[h]_q$ points it contains. The resulting multiset of points defines a linear $\bigl[[h]_qn,k,q^{h-1}d\bigr]_q$ code in which the weight of every codeword is divisible by $q^{h-1}$ \cite[Lemma~4]{kurz2024additive}, and the known bounds for linear codes can be applied to it; see \cite[Definitions~4 and~7]{kurz2024additive}, \cite{kurz2024optimal} and \cite[Theorem~9]{ball2025griesmer}.

We now turn to the linear programming bound. The idea is simple. The numbers of codewords of each weight in a code and in its dual are related by the MacWilliams identities, which are linear equations. So if a code with given parameters exists, these numbers are a nonnegative solution of a certain system of linear equations, and if the system has no such solution, the code does not exist. This linear program can be solved exactly by computer.

To write down the system, set $Q:=q^h$ and, for $0\leq j\leq n$, let
$$
K_j(x):=\sum\limits_{\ell=0}^j(-1)^\ell(Q-1)^{j-\ell}\binom{x}{\ell}\binom{n-x}{j-\ell}
$$
be the $Q$-ary Krawtchouk polynomials. For an additive code $\mC\subseteq\F_{q^h}^n$ of $\F_q$-dimension $k$, let $A_i$ be the number of codewords of weight $i$, and let $B_j$ be the number of elements of $\widetilde\mC^\perp$ of folded Hamming weight $j$. The MacWilliams identities for additive codes \cite{macwilliams1977theory,delsarte1973}, in the form of \cite[Theorem~5]{martinez2025folded}, state that for $0\leq j\leq n$, we have
$$
q^kB_j=\sum\limits_{i=0}^nA_iK_j(i).
$$
Hence, if $d(\mC)\geq d$ and $d^\perp(\mC)\geq d^\perp$, the numbers $A_i$ and $B_j$ are nonnegative integers satisfying
\begin{equation}\label{eq:additive-lp}
\begin{gathered}
A_0=B_0=1,\qquad\sum\limits_iA_i=q^k,\qquad A_i=0\ \ (1\leq i<d),\qquad B_j=0\ \ (1\leq j<d^\perp),\\
q^kB_j=\sum\limits_iA_iK_j(i)\ \ (0\leq j\leq n).
\end{gathered}
\end{equation}
We use this system in two ways. For the \textbf{linear programming bound} (LP) we only ask the $A_i$ and $B_j$ to be nonnegative real numbers. For its \textbf{integer version} (ILP) we ask them to be nonnegative integers, as they are for an actual code. In both cases, if there is no solution, then there is no code with these parameters. The integer version can exclude more parameters, but it is harder to solve. Since $B_j=q^{-k}\sum_iA_iK_j(i)$, the conditions $B_j\geq0$ are the inequalities $\sum_iA_iK_j(i)\geq0$. Without the conditions $B_j=0$ coming from the dual distance, these are Delsarte's linear programming constraints \cite{delsarte1973}.

By \cref{thm:distance-connectivity,thm:dual-berge-girth}, a faithful hypergraphic code with a connected underlying hypergraph $H$ on $N$ vertices and $m$ edges gives a solution of \eqref{eq:additive-lp} with $n=m$, $k=N-1$, $d=\lambda(H)$ and $d^\perp=g_{\rm B}(H)$. Since the code $\mC_H$ exists over every field, we obtain the following test: if, for some prime power $q$, \eqref{eq:additive-lp} has no solution with $n=m$, $k=N-1$ and given $d$ and $d^\perp$, then no connected $(h+1)$-uniform hypergraph with $N$ vertices and $m$ edges satisfies $\lambda(H)\geq d$ and $g_{\rm B}(H)\geq d^\perp$. \Cref{subsec:numerics} shows how these bounds compare with the ones coming from the hypergraph.

\subsection{A Laplacian lower bound on the distance}\label{sec:spectral}

Recall from \cref{subsec:hypergraphs} that the weighted $2$-section of $H$ is the graph on $V$ in which two distinct vertices $u$ and $v$ are joined with weight $a_{uv}$, the number of edges of $H$ containing both, and that $L_H$ is its Laplacian, with eigenvalues $0=\mu_1\leq\mu_2\leq\cdots\leq\mu_N$. We use $\mu_2$ to bound the minimum distance of the associated code from below. For graphs, the corresponding bound $|\delta_{\mG}(S)|\geq\mu_2|S|(N-|S|)/N$ is classical; see e.g. \cite{fiedler1973algebraic,brouwer2012spectra}.

\begin{theorem}\label{thm:spectral-lower}
Let $H$ be a connected hypergraph with $N\geq2$ vertices whose edges have at most $r\geq2$ vertices, and let $\mu_2$ be the second smallest eigenvalue of the Laplacian $L_H$ of its weighted $2$-section. For every $\emptyset\neq S\subsetneq V$ we have
$$
|\delta_H(S)|\geq\frac{\mu_2}{\lfloor r^2/4\rfloor}\cdot\frac{|S|(N-|S|)}{N}.
$$
Consequently, every hypergraphic additive code $\mC$ with underlying hypergraph $H$ satisfies $d(\mC)=\lambda(H)\geq\bigl\lceil\mu_2(N-1)/(N\lfloor r^2/4\rfloor)\bigr\rceil$.
\end{theorem}

\begin{proof}
Let $x\in\mathbb R^V$ be the indicator vector of $S$, that is, $x_v=1$ if $v\in S$ and $x_v=0$ otherwise. Since $x^\top L_Hx=\sum\limits_{\{u,v\}}a_{uv}(x_u-x_v)^2$, we have
$$
x^\top L_Hx=\sum\limits_{u\in S,\ v\notin S}a_{uv},
$$
which is the number of pairs $(e,\{u,v\})$ where $e$ is an edge, $u\in e\cap S$ and $v\in e\setminus S$. An edge $e$ with $|e\cap S|=j$ occurs in $j(|e|-j)$ such pairs. This number is $0$ unless $e\in\delta_H(S)$, and it is at most $\lfloor|e|^2/4\rfloor\leq\lfloor r^2/4\rfloor$. Hence
$$
x^\top L_Hx\leq\lfloor r^2/4\rfloor\,|\delta_H(S)|.
$$
For a lower bound, let $y:=x-\frac{|S|}{N}\mathbf 1$. Since $L_H\mathbf 1=0$, we have $y^\top L_Hy=x^\top L_Hx$, and a direct computation gives $\sum_vy_v=0$ and $\|y\|^2=|S|(N-|S|)/N$. The matrix $L_H$ is symmetric, so it has an orthonormal basis of eigenvectors $u_1,\ldots,u_N$ with $L_Hu_i=\mu_iu_i$, and we may take $u_1=\mathbf 1/\sqrt N$. Since $y$ is orthogonal to $u_1$, we can write $y=\sum_{i\geq2}c_iu_i$, and then
$$
y^\top L_Hy=\sum\limits_{i\geq2}\mu_ic_i^2\geq\mu_2\sum\limits_{i\geq2}c_i^2=\mu_2\|y\|^2=\mu_2\frac{|S|(N-|S|)}{N}.
$$
Combining the two estimates gives the inequality. Finally, $|S|(N-|S|)\geq N-1$ whenever $1\leq|S|\leq N-1$, and $d(\mC)=\lambda(H)$ by \cref{thm:distance-connectivity}.
\end{proof}

\begin{example}\label{ex:spectral-examples}
For the Fano plane (\cref{ex:fano}), $\mu_2=7$ and $r=3$, so $d(\mC_H)\geq\lceil7\cdot6/(7\cdot2)\rceil=3$. This bound is sharp by \cref{ex:distance-examples}. For $H_0$ (\cref{ex:small-hypergraph}), $\mu_2=5$ gives $d(\mC)\geq\lceil5\cdot3/(4\cdot2)\rceil=2=\lambda(H_0)$. For $H=K_4^{(3)}$, every pair of vertices lies in exactly two edges, so $L_H=2(4I-J)$ with eigenvalues $0,8,8,8$, and \cref{thm:spectral-lower} gives $d(\mC_H)\geq\lceil 8\cdot3/(4\cdot2)\rceil=3$. Together with the Singleton bound $d(\mC_H)\leq4-\lceil3/2\rceil+1=3$, this shows again that $\mC_H$ is QMDS (\cref{ex:qmds-examples}).
\end{example}

The Fano plane is an instance of the following general statement.

\begin{corollary}\label{cor:sts}
Let $H$ be a Steiner triple system on $N\geq7$ points, that is, a $3$-uniform hypergraph in which every pair of vertices lies in exactly one edge. Then $\lambda(H)=(N-1)/2$, and $\mC_H$ is an $\bigl[N(N-1)/6,(N-1)/2,(N-1)/2\bigr]_q^2$ code with $d^\perp(\mC_H)=3$ for every prime power $q$. It attains the bound of \cref{prop:moore}(i) with $t=1$.
\end{corollary}

\begin{proof}
Every vertex lies in $(N-1)/2$ edges, and $H$ has $N(N-1)/6$ edges. The weighted $2$-section is $K_N$ with all weights equal to $1$, so $L_H=NI-J$ and $\mu_2=N$. \Cref{thm:spectral-lower} with $r=3$ gives $\lambda(H)\geq\lceil(N-1)/2\rceil=(N-1)/2$, and \eqref{eq:lambda-delta} gives the reverse inequality. The hypergraph $H$ is linear, and if $a,b,c$ are three vertices not in a common edge, the three edges through two of them form a Berge cycle of length $3$, so $g_{\rm B}(H)=3$. The parameters now follow from \cref{thm:distance-connectivity,thm:dual-berge-girth}, and $k=N-1=h\,d(\mC_H)$ with $h=2$.
\end{proof}

\subsection{Comparison on examples}\label{subsec:numerics}

We now compare the bounds of this paper with the general bounds for additive codes on a list of examples. All values in the tables below were computed with SageMath. The SageMath program used for the computations in this section, together with its output, is available at \url{https://github.com/gianiraalfarano/hypergraphic_additive_codes}.

\begin{example}\label{ex:catalogue}
Besides $H_0$, the complete hypergraphs $K_N^{(r)}$, the Fano plane, the Pasch configuration and the hypergraph $H_1$ of \cref{ex:small-hypergraph,ex:fano,ex:qmds-examples}, the tables below use the following classical incidence structures, with the points as vertices and the lines as edges.
\begin{enumerate}[label=\textup{(\alph*)}]
\item The affine plane $\mathrm{AG}(2,3)$, with the $9$ points of $\F_3^2$ and its $12$ lines, and the projective plane $\PG(2,3)$, with $13$ points and $13$ lines of size $4$; see \cite{colbourn2007handbook}.
\item The Desargues configuration, whose points are the $2$-subsets and whose lines are the $3$-subsets of $[5]$, with inclusion as incidence; see \cite{grunbaum2009configurations}.
\item The generalised quadrangles $\mathrm{GQ}(s,t)$, with $s+1$ points on each line and $t+1$ lines through each point \cite{payne2009finite}: $\mathrm{GQ}(2,1)$ is the $3\times3$ grid, whose lines are the rows and the columns; $\mathrm{GQ}(2,2)$ consists of the points of $\PG(3,2)$ and the lines that are totally isotropic for the symplectic form $x_0y_1+x_1y_0+x_2y_3+x_3y_2$; $\mathrm{GQ}(2,4)$ consists of the points and lines of the elliptic quadric $x_0x_1+x_2x_3+x_4^2+x_4x_5+x_5^2=0$ in $\PG(5,2)$; and $\mathrm{GQ}(4,2)$ is the dual of $\mathrm{GQ}(2,4)$.
\item The split Cayley hexagon $\mathrm{GH}(2,2)$ \cite{vanmaldeghem1998generalized}: its points are the $63$ points of the parabolic quadric $X_0X_4+X_1X_5+X_2X_6=X_3^2$ in $\PG(6,2)$, and its lines are the $63$ lines of this quadric whose Grassmann coordinates $p_{ij}=x_iy_j-x_jy_i$ satisfy $p_{12}=p_{34}$, $p_{54}=p_{32}$, $p_{20}=p_{35}$, $p_{65}=p_{30}$, $p_{01}=p_{36}$ and $p_{46}=p_{31}$.
\end{enumerate}
\end{example}

\medskip\noindent\emph{Minimum distance of hypergraphic codes.} \Cref{tab:distance-hypergraphic} contains the faithful hypergraphic codes $\mC_H$ of a list of connected uniform hypergraphs $H$. For each of them we give the true values $d^\perp=g_{\rm B}(H)$ and $d=\lambda(H)$, the lower bound of \cref{thm:spectral-lower}, and five upper bounds on $d$:
\begin{itemize}
\item the minimum degree, by \eqref{eq:lambda-delta};
\item the girth--distance bound, that is, the largest $d$ allowed by \cref{prop:moore}(i) and (ii) for the given $k=N-1$ and $d^\perp$;
\item the Singleton bound \eqref{eq:singleton};
\item the additive Griesmer bound (\cref{thm:additive-griesmer});
\item the linear programming bound, that is, the largest $d$ for which \eqref{eq:additive-lp} with $n=m$, $k=N-1$ and $d^\perp=g_{\rm B}(H)$ has a solution in nonnegative real numbers.
\end{itemize}
The last two bounds depend on $q$, and we give them for $q=2$. The linear programming column does not change if we replace $g_{\rm B}(H)$ by $2$, that is, if we only use that the code is faithful. In every row $\tau(H)=\lceil(N-1)/h\rceil$, so the bound $m-\tau(H)+1$ of \cref{prop:spanning-singleton} is the Singleton bound.

\begin{table}[ht]
\centering
\small
\setlength{\tabcolsep}{3.6pt}
\begin{tabular}{@{}llrrrrrrrrrrr@{}}
\toprule
& & & & & & & lower & \multicolumn{5}{c}{upper bounds on $d$}\\
\cmidrule(l){8-8}\cmidrule(l){9-13}
$H$ & defined in & $h$ & $N$ & $m$ & $d^\perp$ & $d$ & Thm.~\labelcref{thm:spectral-lower} & \eqref{eq:lambda-delta} & Prop.~\labelcref{prop:moore} & \eqref{eq:singleton} & Thm.~\labelcref{thm:additive-griesmer} & \eqref{eq:additive-lp}\\
\midrule
$H_0$ & Ex.~\labelcref{ex:small-hypergraph} & 2 & 4 & 3 & 2 & 2 & 2 & 2 & -- & 2 & 2 & 2\\
$K_4^{(3)}$ & Ex.~\labelcref{ex:qmds-examples} & 2 & 4 & 4 & 2 & 3 & 3 & 3 & -- & 3 & 3 & 3\\
$K_5^{(3)}$ & Sec.~\labelcref{subsec:hypergraphs} & 2 & 5 & 10 & 2 & 6 & 6 & 6 & -- & 9 & 8 & 8\\
$K_6^{(3)}$ & Sec.~\labelcref{subsec:hypergraphs} & 2 & 6 & 20 & 2 & 10 & 10 & 10 & -- & 18 & 15 & 15\\
Pasch configuration & Ex.~\labelcref{ex:qmds-examples} & 2 & 6 & 4 & 3 & 2 & 2 & 2 & 2 & 2 & 2 & 2\\
Fano plane & Ex.~\labelcref{ex:fano} & 2 & 7 & 7 & 3 & 3 & 3 & 3 & 3 & 5 & 4 & 4\\
$\mathrm{AG}(2,3)$ & Ex.~\labelcref{ex:catalogue} & 2 & 9 & 12 & 3 & 4 & 4 & 4 & 4 & 9 & 8 & 7\\
Desargues configuration & Ex.~\labelcref{ex:catalogue} & 2 & 10 & 10 & 3 & 3 & 3 & 3 & 4 & 6 & 6 & 6\\
$\mathrm{GQ}(2,1)$ & Ex.~\labelcref{ex:catalogue} & 2 & 9 & 6 & 4 & 2 & 2 & 2 & 2 & 3 & 3 & 2\\
$\mathrm{GQ}(2,2)$ & Ex.~\labelcref{ex:catalogue} & 2 & 15 & 15 & 4 & 3 & 3 & 3 & 3 & 9 & 8 & 7\\
$\mathrm{GQ}(2,4)$ & Ex.~\labelcref{ex:catalogue} & 2 & 27 & 45 & 4 & 5 & 5 & 5 & 5 & 33 & 28 & 25\\
$\mathrm{GH}(2,2)$ & Ex.~\labelcref{ex:catalogue} & 2 & 63 & 63 & 6 & 3 & 2 & 3 & 3 & 33 & 32 & 24\\
$H_1$ & Ex.~\labelcref{ex:qmds-examples} & 3 & 6 & 3 & 2 & 2 & 2 & 2 & -- & 2 & 2 & 2\\
$K_5^{(4)}$ & Ex.~\labelcref{ex:qmds-examples} & 3 & 5 & 5 & 2 & 4 & 3 & 4 & -- & 4 & 4 & 4\\
$\PG(2,3)$ & Ex.~\labelcref{ex:catalogue} & 3 & 13 & 13 & 3 & 4 & 3 & 4 & 4 & 10 & 10 & 9\\
$\mathrm{GQ}(4,2)$ & Ex.~\labelcref{ex:catalogue} & 4 & 45 & 27 & 4 & 3 & 2 & 3 & 3 & 17 & 23 & 16\\
\bottomrule
\end{tabular}
\caption{Faithful hypergraphic $[m,(N-1)/h,d]_q^h$ codes $\mC_H$ with $d=\lambda(H)$ and $d^\perp=g_{\rm B}(H)$ (\cref{thm:distance-connectivity,thm:dual-berge-girth}), compared with the lower bound of \cref{thm:spectral-lower} and with the upper bounds \eqref{eq:lambda-delta}, \cref{prop:moore}, \eqref{eq:singleton}, \cref{thm:additive-griesmer} and the linear programming bound from \eqref{eq:additive-lp}. The last two columns are for $q=2$; a dash means that \cref{prop:moore} gives no restriction because $d^\perp\leq2$.}
\label{tab:distance-hypergraphic}
\end{table}

In every row $\lambda(H)=\delta_{\min}(H)$, so the degree bound \eqref{eq:lambda-delta} is attained. The girth--distance bound gives the exact value of $d$ for every linear hypergraph in the table except the Desargues configuration, and the spectral lower bound is attained in $12$ of the $16$ rows, including the Steiner triple systems (see \cref{cor:sts}) and the three generalised quadrangles with $h=2$. On the other hand, the Singleton, Griesmer and linear programming bounds give the true value of $d$ only for the five QMDS codes of the table and, for the linear programming bound, for the $3\times3$ grid, and they are never better than the degree bound. So for hypergraphic codes the useful bounds are the ones that come from the hypergraph.

\medskip\noindent\emph{Critical exponent.} \Cref{tab:crit-hypergraphic} compares $\crit(\mC_H)$ for $q=2$ with the lower bound of \cref{cor:crit-hypergraph-bounds} and with the upper bounds of \cref{cor:crit-hypergraph-bounds,prop:crit-general-bounds,cor:kung-hypergraph}. We omit the lower bound $\lceil n/w_{\max}\rceil$ of \cref{prop:crit-general-bounds}, since by \cref{rem:wmax} it is equal to $\min\{\crit(\mC),2\}$. The bounds coming from colourings and the counting bound $t_0$ are close to the true value, whereas the Kung-type bound grows linearly with $N$. For $K_N^{(3)}$ we have $\alpha_{\mathrm w}=2$ and $\chi_{\mathrm w}=\lceil N/2\rceil$, so the lower bound of \cref{cor:crit-hypergraph-bounds} is attained.

\begin{table}[ht]
\centering
\footnotesize
\setlength{\tabcolsep}{3pt}
\begin{tabular}{@{}llrrrrrrrr@{}}
\toprule
& & & & & lower & \multicolumn{4}{c}{upper bounds on $\crit(\mC_H)$}\\
\cmidrule(l){6-6}\cmidrule(l){7-10}
$H$ & defined in & $N$ & $\chi_{\mathrm w}(H)$ & $\crit(\mC_H)$ & Cor.~\labelcref{cor:crit-hypergraph-bounds} & Cor.~\labelcref{cor:crit-hypergraph-bounds} & Prop.~\labelcref{prop:crit-general-bounds} & Cor.~\labelcref{cor:kung-hypergraph} & Prop.~\labelcref{prop:crit-general-bounds}\\
& & & & & $\alpha_{\mathrm w}$ & $\Delta$ & $t_0$ & & $k-h+1$\\
\midrule
$K_5^{(3)}$ & Sec.~\labelcref{subsec:hypergraphs} & 5 & 3 & 2 & 2 & 3 & 2 & -- & 3\\
$K_6^{(3)}$ & Sec.~\labelcref{subsec:hypergraphs} & 6 & 3 & 2 & 2 & 4 & 3 & -- & 4\\
$K_9^{(3)}$ & Sec.~\labelcref{subsec:hypergraphs} & 9 & 5 & 3 & 3 & 5 & 4 & -- & 7\\
Fano plane & Ex.~\labelcref{ex:fano} & 7 & 3 & 2 & 1 & 2 & 2 & 4 & 5\\
$\mathrm{AG}(2,3)$ & Ex.~\labelcref{ex:catalogue} & 9 & 3 & 2 & 2 & 3 & 2 & 6 & 7\\
Desargues configuration & Ex.~\labelcref{ex:catalogue} & 10 & 2 & 1 & 1 & 2 & 2 & 7 & 8\\
$\mathrm{GQ}(2,2)$ & Ex.~\labelcref{ex:catalogue} & 15 & 2 & 1 & 1 & 2 & 2 & 10 & 13\\
$\mathrm{GQ}(2,4)$ & Ex.~\labelcref{ex:catalogue} & 27 & 2 & 1 & 1 & 3 & 3 & 22 & 25\\
$\mathrm{GH}(2,2)$ & Ex.~\labelcref{ex:catalogue} & 63 & 2 & 1 & 1 & 2 & 3 & 54 & 61\\
$\PG(2,3)$, $h=3$ & Ex.~\labelcref{ex:catalogue} & 13 & 2 & 1 & 1 & 3 & 2 & 8 & 10\\
$\mathrm{GQ}(4,2)$, $h=4$ & Ex.~\labelcref{ex:catalogue} & 45 & 2 & 1 & 1 & 2 & 2 & 34 & 41\\
\bottomrule
\end{tabular}
\caption{The critical exponent of hypergraphic codes $\mC_H$ for $q=2$, computed by \cref{thm:chromatic-critical}, with the lower bound $\lceil\log_2\lceil N/\alpha_{\mathrm w}(H)\rceil\rceil$ and the upper bound $\lceil\log_2(\Delta(H)+1)\rceil$ of \cref{cor:crit-hypergraph-bounds}, the bounds $t_0$ and $k-h+1$ of \cref{prop:crit-general-bounds}, and the bound of \cref{cor:kung-hypergraph}; a dash means that \cref{cor:kung-hypergraph} does not apply. Here $h=2$ unless stated otherwise.}
\label{tab:crit-hypergraphic}
\end{table}

\medskip\noindent\emph{General additive codes.} \Cref{tab:general} contains some faithful additive codes that are not of the form $\mC_H$: the codes of spreads (\cref{ex:spread}), the hexacode \cite{huffman2003fundamentals}, regarded as an additive $[6,6/2,4]_2^2$ code, and two codes with parameters $[22,7/2,16]_2^2$ and $[25,7/2,18]_2^2$ found by computer search in \cite[Section~F]{kurz2024additive}. The first of these has the largest possible length among the codes with $k=7$ and $n-d\leq6$; see \cref{tab:lengths}. For all codes in the table the Griesmer and linear programming bounds on $d$ are attained, so they have the largest minimum distance for their length and dimension. For the critical exponent, both $t_0$ and the Kung-type bound of \cref{thm:kung-block} are attained by the spreads, as shown in \cref{ex:spread}; by \cref{rem:kung-extremal}, for $h\geq2$ and $d^\perp=3$ the spreads are the only faithful codes that attain the Kung-type bound.

\begin{table}[ht]
\centering
\footnotesize
\setlength{\tabcolsep}{2.6pt}
\begin{tabular}{@{}llrrrrrrrrrrrrr@{}}
\toprule
& & & & & & & & \multicolumn{3}{c}{upper bounds on $d$} & & \multicolumn{3}{c}{upper bounds on $\crit(\mC)$}\\
\cmidrule(l){9-11}\cmidrule(l){13-15}
code & defined in & $q$ & $h$ & $n$ & $k$ & $d$ & $d^\perp$ & \eqref{eq:singleton} & Thm.~\labelcref{thm:additive-griesmer} & \eqref{eq:additive-lp} & $\crit(\mC)$ & Prop.~\labelcref{prop:crit-general-bounds} & Thm.~\labelcref{thm:kung-block} & Prop.~\labelcref{prop:crit-general-bounds}\\
& & & & & & & & & & & & $t_0$ & & $k-h+1$\\
\midrule
spread & Ex.~\labelcref{ex:spread} & 2 & 2 & 5 & 4 & 4 & 3 & 4 & 4 & 4 & 2 & 2 & 2 & 3\\
spread & Ex.~\labelcref{ex:spread} & 3 & 2 & 10 & 4 & 9 & 3 & 9 & 9 & 9 & 2 & 2 & 2 & 3\\
spread & Ex.~\labelcref{ex:spread} & 2 & 3 & 9 & 6 & 8 & 3 & 8 & 8 & 8 & 2 & 2 & 2 & 4\\
hexacode & \cite{huffman2003fundamentals} & 2 & 2 & 6 & 6 & 4 & 4 & 4 & 4 & 4 & 1 & 2 & 2 & 5\\
search & \cite[Sec.~F]{kurz2024additive} & 2 & 2 & 22 & 7 & 16 & 2 & 19 & 16 & 16 & 2 & 3 & -- & 6\\
search & \cite[Sec.~F]{kurz2024additive} & 2 & 2 & 25 & 7 & 18 & 3 & 22 & 18 & 18 & 2 & 3 & 5 & 6\\
\bottomrule
\end{tabular}
\caption{Some faithful additive $[n,k/h,d]_q^h$ codes that are not of the form $\mC_H$, compared with the upper bounds \eqref{eq:singleton}, \cref{thm:additive-griesmer} and the linear programming bound from \eqref{eq:additive-lp} (with the dual distance of the code) on $d$, and with the bounds of \cref{prop:crit-general-bounds,thm:kung-block} on $\crit(\mC)$; a dash means that \cref{thm:kung-block} does not apply because $d^\perp\leq2$.}
\label{tab:general}
\end{table}

\medskip\noindent\emph{Optimal lengths.} Finally, we consider general faithful additive codes over $\F_4$, that is, $q=h=2$. Let $n_2(k,2;s)$ be the largest length $n$ of a faithful additive code over $\F_4$ of $\F_2$-dimension $k$ with $n-d\leq s$. These numbers were determined for $k\leq8$ in \cite[Theorems~12--15]{kurz2024additive}; see also \cite{kurz2024optimal}. For each $(k,s)$ we compute three upper bounds on $n_2(k,2;s)$: the largest $n$ allowed by the additive Griesmer bound, and the largest $n$ for which \eqref{eq:additive-lp} with $d=n-s$ and $d^\perp=2$ has a solution in nonnegative real numbers (LP) or in nonnegative integers (ILP). For $k\in\{5,6\}$ and $\lceil k/2\rceil-1\leq s\leq12$, the Griesmer bound already gives the exact value $n_2(k,2;s)$, except for $n_2(6,2;3)=9<11$, and neither program does better, so \cref{tab:lengths} only lists $k\in\{7,8\}$. Two variants of the linear program give exactly the same values: dropping the condition $B_1=0$, which gives Delsarte's bound, and adding Delsarte's inequalities for the associated binary linear code of \cite[Lemma~4]{kurz2024additive} described after \cref{thm:additive-griesmer}. For $k=5$, for $k=6$ with $s\geq4$ and for $k=7$ with $s\geq8$, the linear program gives exactly the simple bound $\lfloor s[k]_2/[k-2]_2\rfloor$ of \cite[Lemma~16]{kurz2024additive}, which explains why it is weak there.

\begin{table}[ht]
\centering
\small
\begin{tabular}{@{}rrrrrcrrrr@{}}
\toprule
& \multicolumn{4}{c}{$k=7$} & & \multicolumn{4}{c}{$k=8$}\\
\cmidrule(l){2-5}\cmidrule(l){7-10}
$s$ & Thm.~\labelcref{thm:additive-griesmer} & LP \eqref{eq:additive-lp} & ILP \eqref{eq:additive-lp} & $n_2(7,2;s)$ & & Thm.~\labelcref{thm:additive-griesmer} & LP \eqref{eq:additive-lp} & ILP \eqref{eq:additive-lp} & $n_2(8,2;s)$\\
\midrule
3 & 11 & \textbf{9} & \textbf{9} & 7 & & 9 & \textbf{5} & \textbf{5} & 5\\
4 & 14 & 14 & 14 & 12 & & 12 & \textbf{11} & \textbf{11} & 10\\
5 & 19 & \textbf{18} & \textbf{18} & 17 & & 17 & 17 & 17 & 17\\
6 & 22 & 23 & 22 & 22 & & 22 & 22 & 22 & 18\\
7 & 27 & 27 & 27 & 27 & & 25 & 26 & 26 & 23\\
8 & 32 & 32 & 32 & 32 & & 30 & 30 & 30 & 28\\
9 & 35 & 36 & 36 & 35 & & 33 & 34 & 34 & 33\\
10 & 40 & 40 & 40 & 40 & & 38 & 38 & 38 & 36\\
11 & 43 & 45 & 44 & 43 & & 43 & 43 & \textbf{42} & 40\\
12 & 46 & 49 & 48 & 46 & & 44 & 47 & 47 & 44\\
\bottomrule
\end{tabular}
\caption{Upper bounds on $n_2(k,2;s)$: the largest $n$ allowed by the additive Griesmer bound (\cref{thm:additive-griesmer}), and the largest $n$ for which \eqref{eq:additive-lp} with $d=n-s$ and $d^\perp=2$ has a solution in nonnegative real numbers (LP) or in nonnegative integers (ILP). The last columns give the exact values from \cite[Theorems~14 and~15]{kurz2024additive}. Entries in bold improve on the Griesmer bound.}
\label{tab:lengths}
\end{table}

The linear programming bound improves on the Griesmer bound in four cases, and its integer version in one more case, $(k,s)=(8,11)$. For $(k,s)=(8,3)$ both give the exact value $n_2(8,2;3)=5$, which is attained, for instance, by the $\F_4$-linear $[5,4,2]_4$ parity-check code, regarded as an additive code of $\F_2$-dimension $8$. In all other cases in which $n_2(k,2;s)$ is smaller than the Griesmer bound, the exact value was proved in \cite{kurz2024additive} and the references therein by other methods, for instance with bounds for the associated linear codes, and the programs do not reach it. On their own, the programs can also be weaker than the Griesmer bound, for instance for $(k,s)=(7,12)$, so in practice one takes the smaller of the two values. Asking for integer solutions lowers the value of the linear program in eight of the $42$ cases with $5\leq k\leq8$, but it goes below the Griesmer bound only for $(k,s)=(8,11)$.

The linear programming bound only uses the numbers $A_i$ and $B_j$, and the same holds for the eigenvalue bounds for $t$-independent sets in the Hamming graph \cite{abiad2019kindependence,abiad2025eigenvalue} and for the semidefinite programming bounds of \cite{gijswijt2006nonbinary}. Stronger bounds might come from methods that also use the subspaces of the projective system $\mX(\mC)$ themselves, for instance semidefinite programming over the Grassmannian of $h$-subspaces, in the spirit of \cite{bachoc2013projective}. We leave this for future work.

\medskip

\section*{Acknowledgements}
The author is supported by the Agence Nationale de la Recherche through grant number ANR-24-CPJ1-0075-01, and by Rennes Métropole through a grant AIS.

\medskip

\section*{Declaration on the use of AI}
The author certifies that the ideas and results of the paper are her own. The author used Claude, Anthropic (version Opus 5.5) during the preparation of this manuscript only to review the mathematical arguments, to edit the text, and to check examples computationally. Moreover, the programs in SageMath for the computations of \cref{subsec:numerics} were also developed with the aid of this tool. All mathematical content, claims, and conclusions have been reviewed and verified by the author, who takes full responsibility for them.

\bibliographystyle{abbrv}
\bibliography{references}
\end{document}